%% file: GR-K5-complete.tex
\documentclass[twoside, 10pt]{article}
\usepackage{amssymb, amsmath, mathrsfs, amsthm}
\usepackage{graphicx}
\usepackage{color}
\usepackage{float}

\newcommand{\bd}{\begin{description}}
\newcommand{\ed}{\end{description}}
\newcommand{\bi}{\begin{itemize}}
\newcommand{\ei}{\end{itemize}}
\newcommand{\be}{\begin{enumerate}}
\newcommand{\ee}{\end{enumerate}}
\newcommand{\beq}{\begin{equation}}
\newcommand{\eeq}{\end{equation}}
\newcommand{\beqs}{\begin{eqnarray*}}
\newcommand{\eeqs}{\end{eqnarray*}}

\definecolor{DarkGreen}{rgb}{0.2, 0.6, 0.3}

\newcommand{\labelz}[1]{\label{#1}}

\newtheorem{theorem}{Theorem}
\newtheorem{conjecture}{Conjecture}

\newtheorem{lemma}{Lemma}

\newtheorem{case}{Case}
\newtheorem{subcase}{Subcase}[case]
\newtheorem{subsubcase}{Subcase}[subcase]
\newtheorem{claim}{Claim}
\newtheorem{fact}{Fact}

\begin{document}
\title{\bf Gallai-Ramsey number for $K_{5}$}

\author{Colton Magnant\footnote{Department of Mathematics, Clayton State University, Morrow, GA 30260 USA. {\tt dr.colton.magnant@gmail.com}} \footnote{Center for Mathematics and Interdisciplinary Sciences of Qinghai Province, Xining, Qinghai 810008, China.}, Ingo Schiermeyer\footnote{Technische Universit{\"a}t Bergakademie Freiberg, Institut f{\"u}r Diskrete Mathematik und Algebra, 09596 Freiberg, Germany. {\tt Ingo.Schiermeyer@tu-freiberg.de}}}

\maketitle

\begin{abstract}
Given a graph $H$, the $k$-colored Gallai Ramsey number $gr_{k}(K_{3} : H)$ is defined to be the minimum integer $n$ such that every $k$-coloring of the edges of the complete graph on $n$ vertices contains either a rainbow triangle or a monochromatic copy of $H$. Fox et al.~[J.~Fox, A.~Grinshpun, and J.~Pach. The Erd{\H o}s-Hajnal conjecture for rainbow triangles. J.~Combin.~Theory Ser.~B, 111:75-125, 2015.] conjectured the value of the Gallai Ramsey numbers for complete graphs. Recently, this conjecture has been verified for the first open case, when $H = K_{4}$.

In this paper we attack the next case, when $H = K_5$. Surprisingly it turns out, that the validity of the conjecture depends upon the (yet unknown) value of the Ramsey number $R(5,5)$. It is known that $43 \leq R(5,5) \leq 48$ and conjectured that $R(5,5)=43$~[B.D.~McKay and S.P.~Radziszowski. Subgraph counting identities and Ramsey numbers. J.~Combin.~Theory Ser.~B, 69:193-209, 1997]. If $44 \leq R(5,5) \leq 48$, then Fox et al.'s conjecture is true and we present a complete proof. If, however, $R(5,5)=43$, then Fox et al.'s conjecture is false, meaning that at least one of these two conjectures must be false. For the case when $R(5, 5) = 43$, we show lower and upper bounds for the Gallai Ramsey number $gr_{k}(K_{3} : K_5)$.
\end{abstract}




\section{Introduction}

Given a graph $G$ and a positive integer $k$, the $k$-color Ramsey number $r_{k}(G)$ is the minimum number of vertices $n$ such that every $k$-coloring of the edges of $K_{N}$ for $N \geq n$ must contain a monochromatic copy of $G$. We refer to \cite{MR1670625} for a dynamic survey of known Ramsey numbers. As a restricted version of the Ramsey number, the $k$-color Gallai-Ramsey number $gr_{k}(K_{3} : G)$ is defined to be the minimum integer $n$ such that every $k$-coloring of the edges of $K_{N}$ for $N \geq n$ must contain a either a rainbow triangle or a monochromatic copy of $G$. We refer to \cite{FMO14} for a dynamic survey of known Gallai-Ramsey numbers. In particular, the following was recently conjectured for complete graphs.

\begin{conjecture}[\cite{FGP15}]\labelz{Conj:Fox}
For $k\ge 1$ and $p \geq 3$,
$$
gr_{k}(K_{3} : K_{p}) = \begin{cases}
(r(p) - 1)^{k/2} + 1 & \text{ if $k$ is even,}\\
(p - 1)(r(p) - 1)^{(k - 1)/2} + 1 & \text{ if $k$ is odd.}
\end{cases}
$$
\end{conjecture}

The case where $p = 3$ was actually verified in 1983 by Chung and Graham \cite{MR729784}. A simplified proof was given by Gy\'arf\'as et al.~\cite{GSSS10}.

\begin{theorem}[\cite{MR729784}]\labelz{Thm:grK3}
For $k \geq 1$,
$$
gr_{k}(K_{3} : K_{3}) = \begin{cases}
5^{k/2} + 1 & \text{if $k$ is even,}\\
2\cdot 5^{(k-1)/2} + 1 & \text{if $k$ is odd.}
\end{cases}
$$
\end{theorem}

The next case, where $p = 4$, was proven in \cite{LMSSS17}.

\begin{theorem}\labelz{Thm:grK4}
For $k\ge 1$,
$$
gr_{k}(K_{3} : K_{4}) = \begin{cases} 17^{k/2} + 1 & \text{ if } k \text{ is even,}\\
3\cdot 17^{(k - 1)/2} + 1 & \text{ if } k \text{ is odd.}
\end{cases}
$$
\end{theorem}

Our main result is to essentially prove Conjecture~\ref{Conj:Fox} in the case where $p=5$. This result is particularly interesting since $r(K_{5}, K_{5})$ is still not known. Let $R = r(K_{5}, K_{5}) - 1$ and note that the known bounds on this Ramsey number give us $42 \leq R \leq 47$.

\begin{theorem}\labelz{Thm:Main}
For any integer $k \geq 2$,
$$
gr_{k}(K_{3} : K_{5}) = \begin{cases}
R^{k/2} + 1 & \text{ if $k$ is even,}\\
4 \cdot R^{(k - 1)/2} + 1 & \text{ if $k$ is odd}
\end{cases}
$$
unless $R = 42$, in which case we have
$$
\begin{cases}
gr_{k}(K_{3} : K_{5}) = 43 & \text{ if $k=2$},\\
42^{k/2} + 1 \leq gr_{k}(K_{3} : K_{5}) \leq 43^{k/2} + 1 & \text{ if $k \geq 4$ is even,}\\
169 \cdot 42^{(k-3)/2} + 1 \leq gr_{k}(K_{3} : K_{5}) \leq 4 \cdot 43^{(k-1)/2} + 1
 & \text{ if $k \geq 3$ is odd.}
\end{cases}
$$
\end{theorem}

Theorem~\ref{Thm:Main} is proven in Section~\ref{Sec:MainPf}. Note that if $R = 43$, then Theorem~\ref{Thm:Main} implies that Conjecture~\ref{Conj:Fox} is false.

Also recall the following well known conjecture about the sharp value for the $2$-color Ramsey number of $K_{5}$.

\begin{conjecture}[\cite{MR1438619}]\labelz{Conj:McKay}
$R(K_{5}, K_{5}) = 43$.
\end{conjecture}

By Theorem~\ref{Thm:Main}, it turns out that at least one of Conjecture~\ref{Conj:Fox} or Conjecture~\ref{Conj:McKay} must be false.

In order to prove Theorem~\ref{Thm:Main}, we actually prove a more refined version, stated in Theorem~\ref{Thm:grK5}. Note that Theorem~\ref{Thm:Main} follows from Theorem~\ref{Thm:grK5} by setting $r = k$, $s = 0$ and $t = 0$.

To simplify the notation, we let $c_{1}$ denote the case where $r, s, t$ are all even, $c_{2}$ denote the case where $r, s$ are both even and $t$ is odd, and so on for $c_{3}, \dots, c_{11}$.

\begin{theorem}\labelz{Thm:grK5}
For nonnegative integers $r, s, t$, let $k = r + s + t$. Then
\beqs
& gr_{k}(K_{3} : rK_{5}, sK_{4}, tK_{3}) =
& \begin{cases}
R^{r/2} \cdot 17^{s/2} \cdot 5^{t/2} + 1\\
~ ~ \text{ if $r, s, t$ are even, } (c_{1})\\
2 \cdot R^{r/2} \cdot 17^{s/2} \cdot 5^{(t-1)/2} + 1\\
~ ~ \text{ if $r, s$ are even, and $t$ is odd, } (c_{2})\\
3 \cdot R^{r/2} \cdot 17^{(s-1)/2} + 1\\
~ ~ \text{ if $r$ is even, $s$ is odd, and $t = 0$, } (c_{3})\\
4 \cdot R^{(r-1)/2} + 1\\
~ ~ \text{ if $r$ is odd, and $s = t = 0$, } (c_{4})\\
8 \cdot R^{r/2} \cdot 17^{(s-1)/2} \cdot 5^{(t-1)/2} + 1\\
~ ~ \text{ if $r$ is even, and $s, t$ are odd, } (c_{5})\\
13 \cdot R^{(r-1)/2} \cdot 17^{s/2} \cdot 5^{(t-1)/2} + 1\\
~ ~ \text{ if $r, t$ are odd, and $s$ is even, } (c_{6})\\
16 \cdot R^{r/2} \cdot 17^{(s-1)/2} \cdot 5^{(t - 2)/2} + 1\\
~ ~ \text{ if $r, t$ are even, $t \geq 2$, and $s$ is odd, } (c_{7})\\
24 \cdot R^{(r-1)/2} \cdot 17^{(s-1)/2} \cdot 5^{t/2} + 1\\
~ ~ \text{ if $r, s$ are odd, and $t$ is even, } (c_{8})\\
26 \cdot R^{(r-1)/2} \cdot 17^{s/2} \cdot 5^{(t-2)/2} + 1\\
~ ~ \text{ if $r$ is odd, $s$ is even, $t \geq 2$ is even, } (c_{9})\\
48 \cdot R^{(r-1)/2} \cdot 17^{(s-1)/2} \cdot 5^{(t-1)/2} + 1\\
~ ~ \text{ if $r, s, t$ are odd, } (c_{10})\\
72 \cdot R^{(r-1)/2} \cdot 17^{(s-2)/2} + 1\\
~ ~ \text{ if $r$ is odd, $t = 0$, and $s \geq 2$ is even. } (c_{11})\\
\end{cases}
\eeqs
\end{theorem}

For ease of notation, let $g(r, s, t)$ be the value of $gr_{k}(K_{3} : rK_{5}, sK_{4}, tK_{3})$ claimed above. Also, for each $i$ with $1 \leq i \leq 11$, let $g_{i}(r, s, t) = g(r, s, t) - 1$ in the case where $(c_{i})$ holds.

\section{Preliminaries}

In this section, we recall some known results and provide several helpful lemmas that will be used in the proof. First we state the main tool for looking at colored complete graphs with no rainbow triangle.

\begin{theorem}[\cite{MR0221974}]\labelz{Thm:G-Part}
In any coloring of a complete graph containing no rainbow triangle, there exists a nontrivial partition of the vertices (called a Gallai-partition) such that there are at most two colors on the edges between the parts and only one color on the edges between each pair of parts.
\end{theorem}

In light of this result, a colored complete graph with no rainbow triangle is called a \emph{Gallai coloring} and the partition resulting from Theorem~\ref{Thm:G-Part} is called a \emph{Gallai partition}.

Next recall some useful Ramsey numbers.

\begin{theorem}[\cite{GG}]\labelz{Thm:R35}
$$
R(K_{3}, K_{5}) = 14.
$$
\end{theorem}

\begin{theorem}[\cite{MR1324481}]\labelz{Thm:R45}
$$
R(K_{4}, K_{5}) = 25.
$$
\end{theorem}

Also a general lower bound for Gallai-Ramsey numbers, a special case of the main result in \cite{M18}. We will present a more refined construction later for the purpose of proving Theorem~\ref{Thm:grK5}.

\begin{lemma}[\cite{M18}]\labelz{Lem:LowBnd}
For a connected complete graph $H$ of order $n$ and an integer $k \geq 2$, we have
$$
gr_{k}(K_{3} : H) \geq \begin{cases}
(R(H, H) - 1)^{k/2} + 1 & \text{ if $k$ is even,}\\
(n - 1) \cdot (R(H, H) - 1)^{(k - 1)/2} + 1 & \text{ if $k$ is odd.}
\end{cases}
$$
\end{lemma}


We next present several tables of values which concisely capture computations that will be used throughout the proof. Each cell contains the ratio of the corresponding type $g(r_{1}, s_{1}, t_{1})$ in relation to the order of the whole graph $g(r, s, t)$ in the given case. For example, the top left cell of Table~\ref{Table:1} contains the value of the ratio
$$
\frac{g(r, s, t - 1)}{g(r, s, t)}
$$
in the case $(c_{1})$.

Each row of the following tables represents a case (perhaps with some subcases) and each column represents a Type, one of the referenced inequalities listed above it. In some cells containing two values, these values correspond to the extra assumptions listed in the far right column. The cases marked with $-$ do not occur because of base assumptions. The maximum value in each column yields an upper bound on the ratio for that type over all the cases, and these are displayed in Inequalities~\eqref{Ineq:T1}-\eqref{Ineq:T22}.

\newpage

\include{Tables}



Next we provide several lemmas specific to the proof of Theorem~\ref{Thm:Main} but first some definitions.

We call a part $X$ of a Gallai partition \emph{free}, 
if it contains neither red nor blue edges. We call a part $red$ ($blue$) if it contains red (respectively blue) edges, but no red (blue) copy of a $K_3,$ and no blue (red) edges. Note that these notations do not characterize all parts since clearly a part $X$ might fall into none of these categories.

Let $H$ be a Gallai colored complete graph where red and blue are the colors appearing on edges of the reduced graph. We call such a graph (or part of the partition) $H$ a $(R_i,B_j)$-graph if it contains neither a red copy of $K_i$ nor a blue copy of $K_j$.

Let $w_{i, j}(H) = \frac{|H|}{|G|}$ be the \emph{weight} of $H$ as a subgraph of an $(R_{i}, B_{j})$-graph $G$. 
For convenience, when a part $A$ of a Gallai partition of $H$ is assumed, let $H_{R}$ (and $H_{B}$) denote the sets of vertices in $H \setminus A$ with all red (respectively blue) edges to $A$.

For $(R_{3}, B_{3})$-graphs, we get the following.

\begin{lemma}\labelz{Lem:6.1} 
Let $H$ be an $(R_3,B_3)$-graph, whose parts are either free, red, or blue. Then $w_{5, 5}(H) \leq \frac{6.5}{R}$.
\end{lemma}

\begin{proof}
In order to avoid a red or blue triangle, the graph $H$ has $t \leq 5 = R(3,3)-1$ parts. If all parts are free, then $w_{5, 5}(H) \leq \frac{t}{R} \leq \frac{5}{R}$ by Inequality~\eqref{Ineq:T22}.
Suppose $H$ has a red part (and note that a symmetric argument also works for a blue part). Then $H_R$ is empty and $H_B$ contains no blue edges. If all parts of $H_B$ are free, then $w_{5, 5}(H_{B}) \leq \frac{2}{R}$ by Inequality~\eqref{Ineq:T22} since there can be at most two parts in $H_{B}$ (with all red edges in between them). On the other hand, if $H_B$ contains a red part, then $w_{5, 5}(H_B) \leq \frac{3.25}{R}$ by Inequality~\eqref{Ineq:T21} since $H$ consists of two red parts joined by blue edges. In either case, we have $w_{5, 5}(H) \leq \frac{6.5}{R}$.
\end{proof}

For $(R_{3}, B_{4})$-graphs, we get the following.

\begin{lemma}\labelz{Lem:6.2} 
Let $H$ be a $(R_3,B_4)$-graph, whose parts are either free, red, or blue. Then
\bd
\item{\rm(i)} $w_{5, 5}(H) \leq \frac{9.75}{R}$, and furthermore
\item{\rm(ii)} if $H$ contains no red part, then $w_{5, 5}(H) \leq \frac{9.5}{R}$.
\ed
\end{lemma}

\begin{proof}
Since $H$ is an $(R_3,B_4)$-graph, $H$ has $t \leq 8 = R(3,4)-1$ parts. If all parts are free, then $w_{5, 5}(H) \leq \frac{t}{R} \leq \frac{8}{R}$ by Inequality~\eqref{Ineq:T22}. Suppose first that $H$ contains no red parts. Let $X_1$ be a blue part, so $w_{5, 5}(X_1) = \frac{3.25}{R}$ by Inequality~\eqref{Ineq:T21}. Then $H_R$ is an $(R_{2},B_{4})$-graph and
$H_B$ is an $(R_{3},B_{2})$-graph. Similar arguments as in the proof of Lemma~\ref{Lem:6.1} lead to
$$
w_{5, 5}(H_{B}) \leq \max \left\{3 \cdot \frac{1}{R}, \frac{3.25}{R} + \frac{1}{R}\right\} = \frac{4.25}{R}
$$
and
$$
w_{5, 5}(H_{R}) \leq 2 \cdot \frac{1}{R} = \frac{2}{R}
$$
since there is no red part. This gives
$$
w_{5, 5}(H) \leq \frac{1}{R}\left(3.25 + 4.25 + 2\right) = \frac{9.5}{R}.
$$

Now suppose that $H$ contains a red part $X_1$. Then $H_R$ is empty and $H_B$ is an $(R_{3},B_{3})$-graph. By Inequality~\eqref{Ineq:T21} and Lemma~\ref{Lem:6.1}, we obtain
$$
w_{5, 5}(H) \leq w_{5, 5}(X_1) + w_{5, 5}(H_B) \leq \frac{1}{R}\left(3.25 + 6.5\right) = \frac{9.75}{R}.
$$
\end{proof}

The sharpness of Lemma~\ref{Lem:6.2} is given by the following examples:
\bd
\item{(1)} Three red parts joined by blue edges\\
\item{(2)} Two blue parts joined by red edges, which are joined by blue edges with a red part
\ed

For $(R_{5}, B_{3})$-graphs, we get the following.

\begin{lemma}\labelz{Lem:6.3} 
Let $H$ be an $(R_{5},B_{3})$-graph, whose parts are either free, red or blue. Then
\bd
\item{\rm (i)} $w_{5, 5}(H) \leq \frac{13}{R}$ if $H$ contains only free parts;
\item{\rm (ii)} $w_{5, 5}(H) \leq \frac{13}{R}$ if $H$ contains at least one blue part;
\item{\rm (iii)} $w_{5, 5}(H) \leq \frac{12.25}{R}$ if $H$ contains exactly one red part;
\item{\rm (iv)} $w_{5, 5}(H) \leq \frac{14.5}{R}$ if $H$ contains at least two red parts but no two red parts joined by blue edges;
\item{\rm (v)} $w_{5, 5}(H) \leq \frac{13.5}{R}$ if $H$ contains exactly two red parts and they are joined by blue edges;
\item{\rm (vi)} $w_{5, 5}(H) \leq \frac{16.25}{R}$.
\ed
\end{lemma}

\begin{proof}
Since $H$ is an $(R_{5}, B_{3})$-graph, $H$ has $t \leq 13 = R(5,3)-1$ parts. If all parts are free, then $w_{5, 5}(H) \leq \frac{t}{R} \leq \frac{13}{R}$ by Inequality~\eqref{Ineq:T22}. This proves (i) and means that we may assume that $H$ contains at least one red or blue part.

Suppose first that $H$ contains a blue part $X_1$, so $w_{5, 5}(X_1) \leq \frac{3.25}{R}$ by Inequality~\eqref{Ineq:T21}. Then $H_B$ is empty and $H_R$ is an $(R_{4},B_{3})$-graph. By Lemma~\ref{Lem:6.2} (and symmetry of red and blue) we obtain
$$
w_{5, 5}(H) = w_{5, 5}(X_1) + w_{5, 5}(H_R) \leq \frac{1}{R}\left(3.25 + 9.75\right) = \frac{13}{R}.
$$
This shows (ii) and means that we may assume $H$ contains no blue parts for the remainder of the proof.

Suppose next that $H$ contains exactly one red part $X_{1}$. Then $H_R$ has at most $R(3,3)-1 = 5$ parts, $H_B$ has at most $R(5,2)-1 = 4$ parts, and each of these parts must be free. By Inequalities~\eqref{Ineq:T21} and~\eqref{Ineq:T22}, this gives
$$
w_{5, 5}(H) \leq \frac{1}{R}\left(3.25 + 5 + 4\right) = \frac{12.25}{R},
$$
confirming (iii).

Next suppose there are at least two red parts $X_{1}$ and $X_{2}$ but no two red parts joined by blue edges. With only red edges between the red parts, there can only be two such parts. Then $H_{R}$ (with respect to $X_{1}$) contains only free parts other than $X_{2}$ and $H_{B}$ also contains only free parts. As in (iii), $H_R$ has at most $R(3,3)-1 = 5$ parts, $H_B$ has at most $R(5,2)-1 = 4$ parts. By Inequalities~\eqref{Ineq:T21} and~\eqref{Ineq:T22}, this gives
\beqs
w_{5, 5}(H) & = & w_{5, 5}(X_{1}) + w_{5, 5}(X_{2}) + w_{5, 5}(H_{R} \setminus X_{2}) + w_{5, 5}(H_{B})\\
~ & \leq & \frac{1}{R} (3.25 + 3.25 + 4 + 4)\\
~ & = & \frac{14.5}{R},
\eeqs
confirming (iv).

Next suppose there are exactly two red parts $X_{1}$ and $X_{2}$ and they are joined by blue edges. Then $H_{R}$ (with respect to $X_{1}$) contains only free parts and at most $R(3, 3) - 1 = 5$ of them and $H_{B}$ contains only free parts other than $X_{2}$ but no blue edges at all so there can be at most $3$ total parts in $H_{B}$. By Inequalities~\eqref{Ineq:T21} and~\eqref{Ineq:T22}, this means
\beqs
w_{5, 5}(H) & = & w_{5, 5}(X_{1}) + w_{5, 5}(H_{R}) + w_{5, 5}(X_{2}) + w_{5, 5}(H_{B} \setminus X_{2})\\
~ & \leq & \frac{1}{R} (3.25 + 5 + 3.25 + 2)\\
~ & = & \frac{13.5}{R},
\eeqs
confirming (v).

Finally let $X_{1}$ be a red part. Then $H_{B}$ contains no blue edges so it contains at most $4$ free parts, one red part and at most $2$ free parts, or two red parts. By Inequalities~\eqref{Ineq:T21} and~\eqref{Ineq:T22}, this means that
$$
w_{5, 5}(H_{B}) \leq \max\left\{ \frac{4}{R}, \frac{3.25 + 2}{R}, \frac{2\cdot 3.25}{R} \right\} = \frac{6.5}{R}.
$$
On the other side, $H_{R}$ contains no red or blue triangle so it has at most $5$ free parts, one red part and at most $2$ free parts, or two red parts (joined by blue edges). By Inequalities~\eqref{Ineq:T21} and~\eqref{Ineq:T22}, this means that
$$
w_{5, 5}(H_{R}) \leq \max\left\{ \frac{5}{R}, \frac{3.25 + 2}{R}, \frac{2\cdot 3.25}{R} \right\} = \frac{6.5}{R}.
$$
Finally, we have
$$
w_{5, 5}(H) \leq w_{5, 5}(X_{1}) + w_{5, 5}(H_{B}) + w_{5, 5}(H_{R}) \leq \frac{3.25 + 6.5 + 6.5}{R} = \frac{16.25}{R},
$$
confirming (vi).
\end{proof}


\begin{lemma}\labelz{Lem:5.1} 
Let $H$ be an $(R_3,B_3)$-graph, whose parts are either free, red or blue. Then $w_{4, 5}(H) \leq \frac{2}{9}.$
\end{lemma}

\begin{proof}
To avoid a red or blue triangle, $H$ must have at most $R(3, 3) - 1 = 5$ parts. If all these parts are free, then by Inequality~\eqref{Ineq:T16}, we have $w_{4, 5}(H) \leq \frac{5}{24}$. If $H$ has a red or blue part $X_{1}$, say red, then $H_{R}$ is empty and $H_{B}$ contains no blue edges. If all parts in $H_{B}$ are free, then $w_{4, 5}(H_{B}) \leq \frac{1}{12}$ by Inequality~\eqref{Ineq:T16}. On the other hand, if $H_{B}$ contains a red part, then $w_{4, 5}(H_{B}) \leq \frac{1}{9}$ by Inequality~\eqref{Ineq:T15}. In either case, we have $w_{4, 5}(H) \leq \frac{2}{9}$.
\end{proof}

\begin{lemma}\labelz{Lem:5.2} 
Let $H$ be an $(R_3, B_4)$-graph, whose parts are either free, red or blue. Then
\bd
\item{\rm (i)} $w_{4, 5}(H) \leq \frac{25}{72}$ if $H$ contains exactly two blue parts, no red part, and the reduced graph of $H$ is the unique $2$-coloring of $K_{5}$ containing no monochromatic triangle (see Figure~\ref{Fig:Lem5.2}), or 
\item{\rm (ii)} $w_{4, 5}(H) \leq \frac{1}{3}$ otherwise.
\ed
\end{lemma}

\begin{figure}[H]
\begin{center}
\includegraphics{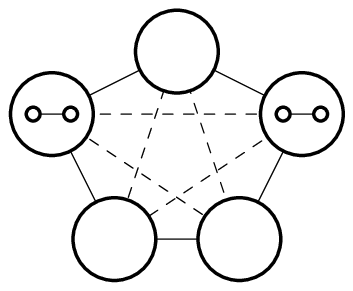}
\caption{The structure of $H$ with solid edges being blue and dashed edges being red \label{Fig:Lem5.2}}
\end{center}
\end{figure}

\begin{proof}
This proof is very similar to the proof of Lemma~\ref{Lem:6.2}. Since $H$ is an $(R_3, B_4)$-graph, it must have $t \leq 8 = R(3, 4) - 1$ parts. If all parts are free, then by Inequality~\eqref{Ineq:T16}, we have $w_{4, 5}(H) \leq \frac{t}{24} \leq \frac{8}{24} = \frac{1}{3}$.

Suppose first that $H$ contains no red parts but does contain at least one blue part. Let $X_1$ be a blue part, so $w_{4, 5}(X_1) = \frac{1}{9}$ by Inequality~\eqref{Ineq:T15}. Then $H_R$ is an $(R_{2},B_{4})$-graph and $H_B$ is an $(R_{3},B_{2})$-graph. Since $H_{B}$ contains no blue edge and there are no red parts, $H_{B}$ must contain at most $2$ parts and these must be free so $w_{4, 5}(H_{B}) \leq \frac{2}{24}$ by Inequality~\eqref{Ineq:T16}.  Since $H_{R}$ contains no red edge, there can be either at most three parts in $H_{R}$ that are all free (with all blue edges in between them), or one blue part and one free part. By Lemmas~\eqref{Ineq:T15} and~\eqref{Ineq:T16}, this means that
$$
w_{4, 5}(H_{R}) \leq \max\left\{ 3\cdot \frac{1}{24}, \frac{1}{9} + \frac{1}{24}\right\} = \frac{11}{72}.
$$
Putting these together, we get
$$
w_{4, 5}(H) \leq \frac{1}{72}\left(8 + 6 + 11 \right) = \frac{25}{72}.
$$
In fact, this bound is only achievable if $H_{R}$ contains one blue part and one free part since otherwise $w_{4, 5}(H_{R}) \leq \frac{1}{8} = \frac{9}{72}$ meaning that $w_{4, 5}(H) \leq \frac{23}{72} < \frac{1}{3}$. Therefore, in order for $w_{4, 5}(H) > \frac{1}{3}$, $H$ must have the structure pictured in Figure~\ref{Fig:Lem5.2}.

Finally suppose $H$ contains a red part $X_1$. Then $H_R$ is empty and $H_B$ is an $(R_{3},B_{3})$-graph. By Inequality~\eqref{Ineq:T15} and Lemma~\ref{Lem:5.1}, we obtain
$$
w_{4, 5}(H) \leq w_{4, 5}(X_1) + w_{4, 5}(H_B) \leq \frac{1}{9} + \frac{2}{9} = \frac{1}{3}.
$$
\end{proof}

\begin{lemma}\labelz{Lem:5.3} 
Let $H$ be an $(R_5, B_3)$-graph, whose parts are either free, red or blue. Then 
\bd
\item{\rm (i)} $w_{4, 5}(H) \leq \frac{5}{9}$ if $H$ consists of $5$ red parts and the reduced graph of $H$ is the unique $2$-coloring of $K_{5}$ with no monochromatic triangle, or
\item{\rm (ii)} $w_{4, 5}(H) \leq \frac{39}{72}$ otherwise.
\ed
\end{lemma}

\begin{proof}
Since $H$ is an $(R_5, B_3)$-graph, there are at most $R(5, 3) - 1 = 13$ parts in $H$. If all these parts are free, then $w_{4, 5}(H) \leq \frac{13}{24} = \frac{39}{72}$ by Inequality~\eqref{Ineq:T16}.

First suppose $H$ contains a blue part $X_{1}$. Then $H_{B}$ is empty and $H_{R}$ is an $(R_{4}, B_{3})$-graph so by Inequality~\eqref{Ineq:T15} and Lemma~\ref{Lem:5.2}, we get
$$
w_{4, 5}(H) \leq w_{4, 5}(X_{1}) + w_{4, 5}(H_{R}) \leq \frac{1}{9} + \frac{25}{72} = \frac{33}{72}.
$$
We may therefore assume $H$ contains no blue part.

If $H$ contains $5$ red parts as described in the statement then clearly $w_{4, 5}(H) \leq \frac{5}{9}$ so suppose this is not the case. That is, suppose $H$ contains a red part $X_{1}$ and at least one free part, so $H_{B}$ (defined in terms of $X_{1}$) contains no blue edges and $H_{R}$ is an $(R_{3}, B_{3})$-graph. Then $H_{B}$ either contains at most $4$ parts that are all free (with all red edges in between them), or one red part and two free parts, or two red parts. By Inequalities~\eqref{Ineq:T15} and~\eqref{Ineq:T16}, this means
$$
w_{4, 5}(H_{B}) \leq \max \left\{ \frac{4}{24}, \frac{1}{9} + \frac{2}{24}, \frac{2}{9}\right\} = \frac{2}{9}.
$$
If $H_{B}$ is not two red parts, then $w_{4, 5}(H_{B}) \leq \frac{14}{72}$. This together with Lemma~\ref{Lem:5.1} gives
$$
w_{4, 5}(H) \leq w_{4, 5}(X_{1}) + w_{4, 5}(H_{B}) + w_{4, 5}(H_{R}) \leq \frac{1}{9} + \frac{14}{72} + \frac{2}{9} = \frac{38}{72}.
$$
Thus, suppose $H_{B}$ has two red parts (so $w_{4, 5}(H_{B}) \leq \frac{2}{9}$) and $H_{R}$ does not have two red parts. By the proof of Lemma~\ref{Lem:5.1}, we have $w_{4, 5}(H_{R}) \leq \frac{5}{24}$. Putting these together, we get
$$
w_{4, 5}(H) \leq w_{4, 5}(X_{1}) + w_{4, 5}(H_{B}) + w_{4, 5}(H_{R}) \leq \frac{1}{9} + \frac{2}{9} + \frac{5}{24} = \frac{39}{72}.
$$
\end{proof}

\section{Three Colors}\labelz{Sec:3-Col}

In this section, we discuss a lower bound example that leads to a counterexample to either Conjecture~\ref{Conj:Fox} or Conjecture~\ref{Conj:McKay}.

\begin{lemma}\labelz{Lem:C-Example}
There exists a $3$-colored copy of $K_{169}$ which contains no rainbow triangle and no monochromatic copy of $K_{5}$.
\end{lemma}

\begin{proof}
Let $G_{rb}$ be a sharpness example on $13$ vertices for the Ramsey number $R(K_{3}, K_{5}) = 14$ say using colors red and blue respectively. Such an example as $G_{rb}$ is $4$-regular in red and $8$-regular in blue. Similarly, let $G_{rg}$ be a copy of the same graph with all blue edges replaced by green edges. We construct the desired graph $G$ by making $13$ copies of each vertex in $G_{rb}$ and for each set of copies (corresponding to a vertex), insert a copy of $G_{rg}$.
If an edge $uv$ in $G_{rb}$ is red (respectively blue), then all edges in $G$ between the two inserted copies of $G_{rg}$ corresponding to $u$ and $v$ are colored red (respectively blue).
Then $G$ contains no rainbow triangle by construction but also contains no monochromatic $K_{5}$. Since $|G| = 169$, this provides the desired example.
\end{proof}

Note that if $R(K_{5}, K_{5}) = 43$ so $R = 42$, then Conjecture~\ref{Conj:Fox} claims that $gr_{3}(K_{3} : K_{5}) = 169$ but this example refutes this claim. On the other hand, if $R(K_{5} : K_{5}) > 43$, then the conjecture holds for $K_{5}$, as proven in Section~\ref{Sec:MainPf} below.

\section{Proof of Theorem~\ref{Thm:grK5} (and Theorem~\ref{Thm:Main})}\labelz{Sec:MainPf}

Note that the lower bound for Theorem~\ref{Thm:Main} follows from Lemma~\ref{Lem:LowBnd} and was also presented in \cite{FGP15} but the lower bound for Theorem~\ref{Thm:grK5} must be more detailed.

\begin{proof}
For the lower bounds, use the following constructions. For all constructions, we start with an $i$-colored base graph $G_{i}$ (constructed below) and inductively suppose we have constructed an $i$-colored graph $G_{i}$ containing no rainbow triangle an no appropriately colored monochromatic cliques. For each two unused colors requiring a $K_{5}$, we construct $G_{i + 2}$ by making $R$ copies of $G_{i}$, adding all edges in between the copies to form a blow-up of a sharpness example for $r(K_{5}, K_{5})$ on $R$ vertices. For each two unused colors requiring a $K_{4}$, we construct $G_{i + 2}$ by making $17$ copies of $G_{i}$, adding all edges in between the copies to form a blow-up of a sharpness example for $r(K_{4}, K_{4})$ on $17$ vertices. For each two unused colors requiring a $K_{3}$, we construct $G_{i + 2}$ by making $5$ copies of $G_{i}$, adding all edges in between the copies to form a blow-up of the sharpness example for $r(K_{3}, K_{3})$ on $5$ vertices.

The base graphs for this construction are constructed by case as follows.
For Case~$(c_{1})$, the base graph $G_{0}$ is a single vertex.
For Case~$(c_{2})$, the base graph $G_{1}$ is a monochromatic copy of $K_{2}$.
For Case~$(c_{3})$, the base graph $G_{1}$ is a monochromatic copy of $K_{3}$.
For Case~$(c_{4})$, the base graph $G_{1}$ is a monochromatic $K_{4}$.
For Case~$(c_{5})$, the base graph $G_{2}$ is a sharpness example on $8$ vertices for $r(K_{3}, K_{4}) = 9$.
For Case~$(c_{6})$, the base graph $G_{2}$ is a sharpness example on $13$ vertices for $r(K_{3}, K_{5}) = 14$.
For Case~$(c_{7})$, the base graph $G_{3}$ is two copies of a sharpness example for $r(K_{3}, K_{4}) = 9$ with all edges in between the copies having a third color.
For Case~$(c_{8})$, the base graph $G_{2}$ is a sharpness example on $24$ vertices for $r(K_{4}, K_{5}) = 25$.
For Case~$(c_{9})$, the base graph $G_{3}$ is two copies of a sharpness example on $13$ vertices for $r(K_{3}, K_{5}) = 14$.
For Case~$(c_{10})$, the base graph $G_{3}$ is two copies of a sharpness example on $24$ vertices for $r(K_{4}, K_{5}) = 25$ with all edges in between the copies having a third color.
For Case~$(c_{11})$, the base graph $G_{3}$ is three copies of a sharpness example on $24$ vertices for $r(K_{4}, K_{5}) = 25$ with all edges in between the copies having a third color.
These base graphs and the corresponding completed constructions contain no rainbow triangle and no appropriately colored monochromatic cliques.

For the upper bound, let $G$ be a Gallai coloring of $K_{n}$ where $n$ is given in the statement. We prove this result by induction on $3r + 2s + t$, meaning that it suffices to either reduce the order of a desired monochromatic subgraph or eliminate a color. Consider a Gallai partition of $G$ and let $q$ be the number of parts in this partition. Choose such a partition so that $q$ is minimized.

\begin{claim}\labelz{Claim:Smallq}
We may assume that $q \geq 4$.
\end{claim}


\begin{proof}
For a contradiction, suppose $q \leq 3$. If $q = 3$, then the reduced graph is a $2$-colored triangle, which contains two edges of the same color. This means that there is a bipartition of the vertices so that all edges in between have one color, contradicting the minimality of $q$. Thus, assume $q = 2$. Let red be the color between the two sets, $A$ and $B$.

First suppose that red is among the last $t$ colors, so we hope to find a red triangle. To avoid a red triangle, there must be no red edges within $A$ or $B$. By induction on $3r + 2s + t$ and using Inequality~\eqref{Ineq:T1}, we get
$$
|G|  =  |A| + |B| \leq  2 [ g(r, s, t - 1) ] = g(r, s, t) < |G|,
$$
a contradiction.

Next suppose that red is among the middle $s$ colors, so we hope to find a red $K_{4}$. To avoid a red $K_{4}$, only one of $A$ or $B$ can have any red edges. Suppose $A$ is allowed to have red edges so $B$ is not. Then observe that $A$ cannot contain a red triangle as this would also create a red $K_{4}$. Thus, by induction on $3r + 2s + t$ and using Inequalities~\eqref{Ineq:T3} and~\eqref{Ineq:T4} respectively, we get
$$
|G| = |A| + |B| \leq g(r, s - 1, t + 1) + g(r, s - 1, t) < g(r, s, t) < |G|,
$$
a contradiction.

Finally suppose red is among the first $r$ colors, so we hope to find a red $K_{5}$. Supposing that the red clique number within $A$ is at least as large as the red clique number within $B$, we get the following requirements:
\bi
\item $A$ contains no red $K_{4}$, and
\item if $A$ contains a red $K_{3}$, then $B$ contains no red edges.
\ei
These leave only two options:
\be
\item $A$ and $B$ both may contain red edges but no red $K_{3}$, or
\item $A$ contains a red $K_{3}$ (but no red $K_{4}$) and $B$ contains no red edges.
\ee

For the first option, we remove $1$ from $r$ but add $1$ to $t$ within both $A$ and $B$. By induction on $3r + 2s + t$ and using Inequality~\eqref{Ineq:T11}, we get
$$
|G| = |A| + |B| \leq 2 g(r - 1, s, t + 1) < g(r, s, t) < |G|,
$$
a contradiction.

For the second option, we remove $1$ from $r$ in both $A$ and $B$ but add $1$ to $s$ in $A$. By induction on $3r + 2s + t$ and using Inequalities~\eqref{Ineq:T9} and~\eqref{Ineq:T12}, we get
$$
|G| = |A| + |B| \leq g(r - 1, s + 1, t) + g(r - 1, s, t) < g(r, s, t) < |G|,
$$
a contradiction. This completes the proof of Claim~\ref{Claim:Smallq}.
\end{proof}

Let $D$ be the reduced graph of the Gallai partition, with vertices $w_{i}$ corresponding to parts $G_{i}$ of the partition. Let $V_{r}$ denote the set of vertices in $D$ whose corresponding sets in the partition contain at least one red edge and let $V_{b}$ denote the set of vertices in $D$ whose corresponding sets in the partition contain at least one blue edge. 
Let $p_2=|V_r \cap V_b|$ be the number of parts containing at least one red and at least one blue edge, $p_1=|V_r\triangle V_b|$ be the number of parts containing at least one red edge or at least one blue edge but not both, and $p_0=|V(D)\setminus(V_r\cup V_b)|$ be the number of parts with no red or blue edges.

For each vertex $w_{i} \in D$, let $d_{r}(w_{i})$ and $d_{b}(w_{i})$ denote its red and blue degrees respectively within $D$. Then $d_{r}(w_{i}) + d_{b}(w_{i}) = q - 1$ for all $i$. By the choice of the Gallai partition with the smallest number of parts, the following fact is immediate.

\begin{fact}\labelz{Fact:Deg1}
For all $w_{i} \in V(D)$, we have $d_{r}(w_{i}), d_{b}(w_{i}) \geq 1$.
\end{fact}

To avoid a monochromatic copy of $K_{5}$, the following facts follow immediately from the relevant definitions.

\begin{fact}\labelz{Fact:MonoCliques}
If $w_{i} \in D$ is in a red $K_{4} \subseteq D$, then $w_{i} \notin V_{r}$. If $w_{i} \in D$ is in a blue $K_{4} \subseteq D$, then $w_{i} \notin V_{b}$.
\end{fact}

\begin{fact}\labelz{Fact:LargeDegrees}
For all $i$,
$$
d_{r}(w_{i}) \leq 24, \text{ and } d_{b}(w_{i}) \leq 24.
$$
\end{fact}

If a vertex $w_{i} \in D$ has at least $r(3, 5) = 14$ incident edges in red (in $D$), then the neighborhood contains either a red $K_{3}$ or a blue $K_{5}$. Certainly the latter is not an option so the former must occur, meaning that $w_{i}$ is contained in a red $K_{4}$ within $D$. By Fact~\ref{Fact:MonoCliques}, we get the following fact.

\begin{fact}\labelz{Fact:SmallDegrees}
If $d_{r}(w_{i}) \geq 14$, then $w_{i} \notin V_{r}$. If $d_{b}(w_{i}) \geq 14$, then $w_{i} \notin V_{b}$.
\end{fact}

The remainder of the proof is broken into cases based on where red and blue fall in the list of colors relative to the first $r$ colors, the middle $s$ colors, and the last $t$ colors.

\setcounter{case}{0}

\begin{case}\labelz{Case:1}
Both red and blue occur within the last $t$ colors.
\end{case}

In this case, the graph $G$ contains no red or blue triangle. Since $r(K_{3}, K_{3}) = 6$, we find that $4 \leq q \leq 5$. By Fact~\ref{Fact:Deg1}, for every $i$, it follows that $G_{i}$ contains no red or blue edge. This means that every $G_{i}$ is colored with at most $k - 2$ colors with
$$
|G_{i}| \leq g(r, s, t - 2).
$$
By induction and Inequality~\eqref{Ineq:T2},
$$
|G| = \sum_{1 = 1}^{q} |G_{i}| \leq 5 g(r, s, t - 2) \leq g(r, s, t) < |G|,
$$
a contradiction, completing the proof of Case~\ref{Case:1}.

\begin{case}\labelz{Case:2}
Red is among the middle $s$ colors while blue is among the last $t$ colors.
\end{case}

In this case, the graph $G$ contains no red $K_{4}$ and no blue triangle. Since $r(K_{4}, K_{3}) = 9$, we find that $4 \leq q \leq 8$. By Fact~\ref{Fact:Deg1}, for every $i$, it follows that $G_{i}$ contains no blue edge and no red triangle.

If $G_{i}$ contains no red edges for some $i$, then $G_{i}$ is colored with at most $k - 2$ colors with $|G_{i}| \leq g(r, s - 1, t - 1)$ so by induction and Inequality~\eqref{Ineq:T5},
\begin{equation}\labelz{Ineq:Case2.1}
|G_{i}| \leq g(r, s - 1, t - 1) \leq \frac{1}{8} g(r, s, t).
\end{equation}
Next if $G_{i}$ contains at least one red edge, then by Fact~\ref{Fact:Deg1}, there can be no red triangle in $G_{i}$ so $|G_{i}| \leq g(r, s - 1, t)$.
Therefore, by the induction hypothesis and Inequality~\eqref{Ineq:T4}, we have
\begin{equation}\labelz{Ineq:Case2.2}
|G_{i}| \leq g(r, s - 1, t) \leq \frac{1}{3} g(r, s, t).
\end{equation}

By Inequalities~\eqref{Ineq:Case2.1} and~\eqref{Ineq:Case2.2}, we get the key inequality
\begin{equation}\labelz{Ineq:Case2-Key}
|G| \leq \left( p_{1} \frac{1}{3} + p_{0} \frac{1}{8} \right) g(r, s, t). 
\end{equation}
This means that as long as we can show
\begin{equation}\labelz{Ineq:Case2-Desired}
\left( p_{1} \frac{1}{3} + p_{0} \frac{1}{8} \right) \leq 1,
\end{equation}
then we obtain a contradiction by showing $|G| \leq g(r, s, t)$. 
The remainder of this case can be concluded by establishing Inequality~\eqref{Ineq:Case2-Desired}, which follows by the same argument as used in the corresponding case of \cite{LMSSS17}.



\begin{case}\labelz{Case:3}
Both red and blue occur within the middle $s$ colors.
\end{case}

In this case, the graph $G$ contains no red or blue $K_{4}$ and cases $(c_{4})$ and $(c_{9})$ cannot occur since $s \geq 2$. Since $r(K_{4}, K_{4}) = 18$, we find that $4 \leq q \leq 17$.
First some bounds on the orders of the parts $G_{i}$, leading to a counterpart of Inequality~\ref{Ineq:Case2-Key}.


First suppose $G_{i}$ contains no red and no blue edges.
Then by induction and Inequality~\eqref{Ineq:T8}, imply that
\begin{equation}\labelz{Ineq:3.1}
|G_{i}| \leq g(r, s - 2, t) = \frac{1}{17} g(r, s, t).
\end{equation}
Next suppose $G_{i}$ contains no blue edges but contains some red edges.
Then by induction and Inequality~\eqref{Ineq:T7}, we get
\begin{equation}\labelz{Ineq:3.2}
|G_{i}| \leq g(r, s - 2, t + 1) \leq \frac{13}{72} g(r, s, t).
\end{equation}
Finally suppose $G_{i}$ contains both red and blue edges.
Then by induction and Inequality~\eqref{Ineq:T6}, we get
\begin{equation}\labelz{Ineq:3.3}
|G_{i}| \leq g(r, s - 2, t + 2) = \frac{13}{36} g(r, s, t).
\end{equation}


Combining Inequalities~\eqref{Ineq:3.1}, \eqref{Ineq:3.2}, and~\eqref{Ineq:3.3}, we obtain the key inequality
\begin{equation}\labelz{Ineq:Case3Key}
|G| \leq \left(p_{2} \frac{13}{36} + p_{1} \frac{13}{72} + p_{0}\frac{1}{17}\right) g(r,s,t).
\end{equation}

As in Case \ref{Case:2}, if we can show that
\begin{equation}\labelz{Ineq:Case3Desired}
p_{2} \frac{13}{36} + p_{1} \frac{13}{72} + p_{0}\frac{1}{17}\le 1,
\end{equation}
then we will arrive at a contradiction that $|G|\leq g(r,s,t)$. Thus, for the remainder of the proof of this case, it suffices to show Inequality~\eqref{Ineq:Case3Desired}.

Next we derive several facts. Within the red neighborhood of some vertex $w_{i}$ in $R$, there can be no red triangle since otherwise we would have a red $K_{4}$ in $G$. There can also be no blue $K_{4}$ within this neighborhood so that means the red neighborhood of $w_{i}$ (and similarly the blue neighborhood) has at most $r(4, 3) - 1 = 8$ vertices. Formally, we obtain the following fact.

\begin{fact}\labelz{Fact:Deg8}
For all $w_{i} \in V(R)$, we have $d_{r}(w_{i}), d_{b}(w_{i}) \leq 8$.
\end{fact}

If a vertex $w_{i} \in R$ is contained in a red (or blue) triangle, then the part $G_{i}$ cannot contain any red (respectively blue) edges to avoid creating a red (respectively blue) copy of $K_{4}$. The following fact is then immediate.

\begin{fact}\labelz{Fact:Triangle}
If $w_{i}$ is in a red triangle in $R$, then $w_{i} \notin V_{r}$. Similarly if $w_{i}$ is in a blue triangle in $R$, then $w_{i} \notin V_{b}$.
\end{fact}

If $d_{r}(w_{i}) \geq 4$ for some $w_{i} \in V(R)$, then the red neighborhood of $w_{i}$ certainly must contain at least one red edge since otherwise, if all edges were blue, we would have a blue $K_{4}$. Thus $w_i$ is in a red triangle in $R$. A similar observation holds with the roles of red and blue switched. Thus from Fact~\ref{Fact:Triangle}, we obtain the following fact.

\begin{fact}\labelz{Fact:Deg4}
If $d_{r}(w_{i}) \geq 4$ then $w_{i} \notin V_{r}$, and if $d_{b}(w_{i}) \geq 4$ then $w_{i} \notin V_{b}$.
\end{fact}

If two parts $G_{i}$ and $G_{j}$ each contain at least one red edge, say $e_{i}$ and $e_{j}$ respectively, then the edge $w_{i}w_{j}$ in $R$ cannot be red since otherwise the subgraph induced on the vertices of $e_{i} \cup e_{j}$ is a red $K_{4}$. Thus, we obtain the following fact.

\begin{fact}\labelz{Fact:Clique}
The subgraph induced on $V_{r}$ is a blue clique and the subgraph induced on $V_{b}$ is a red clique.
\end{fact}

Next, we prove three helpful claims about the values of $p_{0}, p_{1}$, and $p_{2}$.

\begin{claim}\labelz{Claim:p2<=1}
$p_2=|V_r\cap V_b|\le 1$ and if $p_{2} = 1$, then $q \leq 7$.
\end{claim}

\begin{proof}
If we have $w_i,w_j\in V_{r} \cap V_{b}$, then by Fact \ref{Fact:Clique}, $w_i,w_j\in V_{r}$ implies that the edge $w_iw_j$ is blue in $R$, while $w_i,w_j\in V_{b}$ implies that $w_iw_j$ is red, a contradiction.

Now suppose $p_{0} = 1$ and, for a contradiction, that $q \geq 8$. If $w_{1} \in V_{r} \cap V_{b}$, then there are at least $4$ other vertices, say $W = \{w_{2}, w_{3}, w_{4}, w_{5}\}$ with all one color, say red, on edges to $w_{1}$. Since $w_{1} \in V_{r} \cap V_{b}$ and to avoid a red $K_{4}$, all edges between vertices in $W$ must be blue, forming a blue $K_{4}$ for a contradiction.
\end{proof}

\begin{claim}\labelz{Claim:VbVr4} 
$|V_{r}| + |V_{b}| \leq 4$.
\end{claim}

\begin{proof}
Suppose first that there is a vertex $w_{i} \in V_{r} \cap V_{b}$. Then by Fact \ref{Fact:Triangle}, $w_{i}$ is contained in neither a red triangle nor a blue triangle within $R$. By Fact \ref{Fact:Clique}, any vertex of $V_r\setminus\{w_i\}$ must be a blue neighbor of $w_i$ in $R$, and since the blue neighborhood of $w_i$ induces a red clique in $R$, again Fact \ref{Fact:Clique} implies that there can only be at most one vertex in $V_r\setminus\{w_i\}$. This means that $|V_{r}| \leq 2$, and similarly, $|V_{b}| \leq 2$.

Thus, we may assume $V_{r} \cap V_{b} = \emptyset$. We next claim that $|V_{r}| \leq 3$ and $|V_{b}| \leq 3$. If $|V_{r}| \geq 4$, then by Fact \ref{Fact:Clique}, the subgraph of $R$ induced on the vertices of $V_{r}$ contains a blue $K_{4}$, a contradiction. Thus $|V_{r}| \leq 3$, and symmetrically $|V_{b}| \leq 3$.

Now suppose that $|V_{r}| = |V_{b}| = 3$. If there exists a vertex $w_i\in V_r$ with at least two red neighbors in $V_b$, then by Fact \ref{Fact:Clique}, $w_i$ is in a red triangle in $R$, and this contradicts Fact \ref{Fact:Triangle}. Thus, there can be at most one red edge from each vertex in $V_{r}$ to $V_{b}$, and similarly, at most one blue edge from each vertex in $V_{b}$ to $V_{r}$, for a total of at most $6$ edges. But $R$ has $9$ edges between $V_r$ and $V_b$, a contradiction. Finally suppose $|V_{r}| = 3$ and $|V_{b}| = 2$. Then again, there can be at most one red edge from each vertex of $V_{r}$ to $V_{b}$, and at most one blue edge from each vertex of $V_{b}$ to $V_{r}$, for a total of at most $5$ edges, while $R$ has $6$ edges between $V_r$ and $V_b$, another contradiction. Symmetrically we cannot have $|V_{r}| = 2$ and $|V_{b}| = 3$, thus completing the proof of Claim~\ref{Claim:VbVr4}.
\end{proof}


\begin{claim}\labelz{Claim:VbVr3}
If either $|V_{r}| \geq 2$ or $|V_{b}| \geq 2$, then
\begin{description}
\item{{\rm (a)}} $q \leq 10$,
\item{{\rm (b)}} If $q = 10$, then $p_{2} = 0$ and $p_{1} = 2$, and
\item{{\rm (c)}} If $q = 9$, then $p_{1} + p_{2} \leq 3$.
\end{description}
\end{claim}

In particular, if $|V_{r}| + |V_{b}| \geq 3$, then either $|V_{r}| \geq 2$ or $|V_{b}| \geq 2$ so this claim may be applied.

\begin{proof}
Working under the assumption that $|V_{r}| + |V_{b}| \geq 3$, without loss of generality, we may assume $|V_{r}| \geq 2$. Suppose $G_{1}$ and $G_{2}$ each contain at least one red edge. By Fact~\ref{Fact:Clique}, all edges from $G_{1}$ to $G_{2}$ must be blue. For $i$ with $1 \leq i \leq 2$, define
\beqs
R_{i} & = & \{ j | j \geq 3, \text{ and $G_{j}$ is joined to $G_{i}$ by red edges}\\
~ & ~ & \text{and to $G_{3 - i}$ by blue edges} \}\\
R_{1,2} & = & \{ j | j \geq 3, \text{ and $G_{j}$ is joined to $G_{i}$ and $G_{3 - i}$ by red edges} \}\\
B & = & \{ j | j \geq 3, \text{ and $G_{j}$ is joined to $G_{i}$ and $G_{3 - i}$ by blue edges} \}
\eeqs


If $j_{1}, j_{2} \in R_{i}$, then $G_{j_{1}}$ and $G_{j_{2}}$ are joined by blue edges to avoid a red $K_{4}$. Suppose that $|R_{i}| \geq 3$ for some $i$, say $|R_{1}| \geq 3$ with $\{j_{1}, j_{2}, j_{3}\} \subseteq R_{1}$. Then there is a blue $K_{4}$ with vertices chosen from $G_{2}, G_{j_{1}}, G_{j_{2}}, G_{j_{3}}$, a contradiction. This means that $|R_{i}| \leq 2$ for each $i$ with $1 \leq i \leq 2$.

If $j_{1}, j_{2} \in R_{i} \cup R_{1, 2}$, then again $G_{j_{1}}$ and $G_{j_{2}}$ are joined by blue edges to avoid a red $K_{4}$. Then to avoid a blue $K_{4}$, it is clear that $|R_{1, 2}| + |R_{i}| \leq 3$ for each $i$ with $1 \leq i \leq 2$. Exchanging the roles of the colors, it is also clear that $|B| \leq 3$.

Then if:
\bi
\item $|R_{1, 2}| = 0$, then $|R_{1, 2}| + |R_{1}| + |R_{2}| + |B| \leq 0 + 2\cdot 2 + 3 = 7$,
\item $|R_{1, 2}| = 1$, then $|R_{1, 2}| + |R_{1}| + |R_{2}| + |B| \leq 1 + 2\cdot 2 + 3 = 8$,
\item $|R_{1, 2}| = 2$, then $|R_{1, 2}| + |R_{1}| + |R_{2}| + |B| \leq 2 + 2\cdot 1 + 3 = 7$,
\item $|R_{1, 2}| = 3$, then $|R_{1, 2}| + |R_{1}| + |R_{2}| + |B| \leq 3 + 2\cdot 0 + 3 = 6$,
\ei
so in every case, $q = 2 + |R_{1, 2}| + |R_{1}| + |R_{2}| + |B| \leq 2 + 8 = 10$, completing the proof of (a).


If $q = 10$, then we must have $|R_{1, 2}| + |R_{1}| + |R_{2}| + |B| = 1 + 2\cdot 2 + 3 = 8$ so $|R_{1, 2}| = 1$, $|R_{1}| = |R_{2}| = 2$ and $|B| = 3$. By the observations above, all edges between pairs of parts with indices in $B$ are red, meaning that each of these parts is in a red triangle in $R$. Similarly, all edges between pairs of parts with indices in $R_{1, 2} \cup R_{i}$ are blue for each $i$ with $1 \leq i \leq 2$, meaning that each of these parts is in a blue triangle in $R$. Thus, for all $j$ with $3 \leq j \leq 10$, we have $G_{j}$ contains no red or blue edges. Similarly, with $R_{i} = 2$ for each $i$, there can be no blue edges in either $G_{1}$ or $G_{2}$. This means that $p_{2} = 0$ and $p_{1} = 2$, completing the proof of (b).

Finally suppose $q = 9$. What remains of the proof of Claim~\ref{Claim:VbVr3}, we break into cases based on the value of $|R_{1, 2}|$.

If $|R_{1, 2}| = 0$, then $|R_{1}| = |R_{2}| = 2$ and $|B| = 3$. As in the case when $q = 10$, $G_{1}$ and $G_{2}$ each contain no blue edges and for all $j$ with $3 \leq j \leq 10$, $G_{j}$ must contain no red or blue edges. Thus, $p_{2} = 0$ and $p_{1} = 2$.

If $|R_{1, 2}| = 1$, then it is possible that either $|R_{1}| = 1$ or $|R_{2}| = 1$, say $|R_{1}| = 1$. Then the set $G_{j}$ corresponding to $R_{1}$ can have blue edges but all other sets $G_{j}$ with $j \geq 3$ must have no blue and no red edges. Thus, $p_{1} + p_{2} \leq 3$. On the other hand, if $|R_{1}| = |R_{2}| = 2$, then it is possible that $|B| = 2$. Then at most one of the sets $G_{j}$ corresponding to $B$ can have red edges but all other sets $G_{j}$ with $j \geq 3$ must have no blue and no red edges. Thus, $p_{1} + p_{2} \leq 3$.

If $|R_{1, 2}| = 2$, then $|R_{1}| = |R_{2}| = 1$ and $|B| = 3$. Each of the sets with indices in $R_{1} \cup R_{2} \cup R_{1, 2}$ is contained in a blue triangle in $R$, meaning that for all $j$ with $3 \leq j \leq 10$, the set $G_{j}$ contains no red or blue edges. Each set $G_{1}$ and $G_{2}$ may contain blue edges or not but in either case, $p_{1} + p_{2} = 2$, completing the proof of (c) and Claim~\ref{Claim:VbVr3}.
\end{proof}


We now consider subcases based on the value of $q$.

\begin{subcase}\labelz{Subcase:3.1}
$13 \leq q \leq 17.$
\end{subcase}

By Fact~\ref{Fact:Deg8}, we have $d_{r}(w_{i}), d_{b}(w_{i}) \leq 8$ so this means that $d_{b}(w_{i}), d_{r}(w_{i}) \geq 4$ for all $w_{i} \in V(R)$. This means that $G_{i}$ contains no red or blue edges for all $i$. Thus $p_2=p_1=0$, $p_0=q$, and
\[
p_{2} \frac{16/3}{17} + p_{1} \frac{8/3}{17} + p_{0}\frac{1}{17}=\frac{q}{17}\le 1,
\]
as required for Inequality~\eqref{Ineq:Case3Desired}.

\begin{subcase}\labelz{Subcase:3.2}
$4 \leq q \leq 10$.
\end{subcase}

By Claim \ref{Claim:p2<=1}, we have $p_2\le 1$. First suppose $p_{2} = 1$. Then if $q \geq 8$, every vertex $w_{i} \in V(R)$ must have at least $4$ edges in one color and, by Fact~\ref{Fact:Deg4}, every set $G_{i}$ is missing either red or blue, contradicting the assumption that $p_{2} = 1$. Thus, we have $4 \leq q \leq 7$. By Claim~\ref{Claim:VbVr4}, since $p_{2} = 1$, we have $p_{1} =|V_r|+|V_b|-2p_2\leq 2$. Thus,
$$
p_{2} \frac{13}{36} + p_{1} \frac{13}{72} + p_{0} \frac{1}{17} \leq 1\cdot \frac{13}{36} + 2\cdot \frac{13}{72} + (q - 3)\cdot \frac{1}{17} \leq \frac{1172}{1224} < 1.
$$

Next suppose $p_{2} = 0$ so by Claim~\ref{Claim:VbVr4}, $p_{1} \leq 4$.
If $q \leq 8$, we get
$$
p_{2} \frac{13}{36} + p_{1} \frac{13}{72} + p_{0} \frac{1}{17} \leq 4\cdot \frac{13}{72} + 4\cdot \frac{1}{17} = \frac{1172}{1224} < 1.
$$
If $q = 10$, then by Claim~\ref{Claim:VbVr3}, we have $p_{1} = 2$, so
$$
p_{2} \frac{13}{36} + p_{1} \frac{13}{72} + p_{0} \frac{1}{17} \leq 2\cdot \frac{13}{72} + 8\cdot \frac{1}{17} = \frac{1018}{1224} < 1.
$$
If $q = 9$, then by Claim~\ref{Claim:VbVr3}, we have $p_{1} \leq 3$, so
$$
p_{2} \frac{13}{36} + p_{1} \frac{13}{72} + p_{0} \frac{1}{17} \leq 3\cdot \frac{13}{72} + 6\cdot \frac{1}{17} = \frac{1095}{1224} < 1.
$$

\begin{subcase}\labelz{Subcase:3.3}
$q \in\{11, 12\}$.
\end{subcase}

By Claim~\ref{Claim:VbVr3}, we see that $|V_{r}| + |V_{b}| \leq 2$. By Claim~\ref{Claim:p2<=1}, we have $p_{2} = 0$ so $p_{1} \leq 2$ and
$$
p_{2} \frac{13}{36} + p_{1} \frac{13}{72} + p_{0} \frac{1}{17} \leq 2\cdot \frac{13}{72} + 10\cdot \frac{1}{17} = \frac{1162}{1224} < 1
$$
completing the proof of this subcase, and the proof of Case~\ref{Case:3}.

\begin{case}\labelz{Case:4}
Red is among the first $r$ colors while blue is among the last $t$ colors.
\end{case}


In this case, the graph $G$ contains no red $K_{5}$ and no blue triangle. Since $r(K_{3}, K_{5}) = 14$, we find that $4 \leq q \leq 13$. By Fact~\ref{Fact:Deg1}, each part of the Gallai partition has both red and blue incident edges in the reduced graph. This means that there can be no red $K_{4}$ and no blue edge in any part, leading to the following main subcases.

\be
\item \labelz{SubCase4.1} No part has any red edges,
\item \labelz{SubCase4.3} There is a part with a red $K_{3}$, and
\item \labelz{SubCase4.2} There is a part with red edges but no part has a red $K_{3}$.
\ee


We first consider Subcase~\ref{SubCase4.1}. Since every part $G_{i}$ contains no red or blue edges, this means that $|G_{i}| \leq g(r - 1, s, t - 1)$.
By induction and Inequality~\eqref{Ineq:T13}, we get
$$
|G| = \sum_{i = 1}^{q} |G_{i}| \leq \sum_{i = 1}^{q} g(r - 1, s, t - 1) \leq \frac{q}{13} g(r, s, t) \leq g(r, s, t) < |G|,
$$
a contradiction, completing the proof of Subcase~\ref{SubCase4.1}.

Next we consider Subcase~\ref{SubCase4.3}. Let $G_{1}$ be a part of the Gallai partition containing a red triangle. Partition the remaining vertices of $G$ into $G_{R}$ and $G_{B}$ such that $G_{R}$ contains all vertices in parts having red edges to $G_{1}$ and $G_{B}$ contains all vertices in parts having blue edges to $G_{1}$.

Certainly $G_{R}$ contains no red edges and no blue triangle and $G_{B}$ contains no red $K_{5}$ and no blue edges. This means that
$$
|G_{R}| \leq g(r - 1, s, t) ~ ~ ~ ~ ~ \text{and} ~ ~ ~ ~ ~ |G_{B}| \leq g(r, s, t - 1).
$$
Furthermore, since $G_{1}$ contains a red triangle but no red $K_{4}$ and no blue edges, we get
$$
|G_{1}| \leq g(r - 1, s + 1, t - 1).
$$


By considering each of the cases of the statement ($(c_{1})$ up to $(c_{11})$) individually across Types T1, T10, and T12 in Tables~\ref{Table:1} and~\ref{Table:2}, we see that in each of the cases, we have
$$
g(r - 1, s + 1, t - 1) + g(r - 1, s, t) + g(r, s, t - 1) \leq g(r, s, t)
$$
except in the two cases $(c_{7})$ and $(c_{10})$.
We sharpen the bounds by observing that $G_{1}$ contains a red triangle, so every pair of parts in $G_{R}$ are joined by blue edges. Since $G$ has no blue triangle, this means that $G_{R}$ must have at most two parts. If $G_{R}$ has only one part, then by Fact~\ref{Fact:Deg1}, it must have blue edges to some part in $G_{B}$, and so cannot contain a blue edge so $|G_{R}| \leq g(r - 1, s, t - 1)$. If $G_{R}$ has two parts, then similarly each cannot contain blue edges meaning that $|G_{R}| \leq 2g(r - 1, s, t - 1)$. Then the calculations for these two specific cases become
\begin{description}
\item[$(c_{7})$]: $2 \frac{3}{R} + \frac{221}{16R} + \frac{1}{2} < 1$, and
\item[$(c_{10})$]: $\frac{2}{15} + \frac{17}{48} + \frac{1}{2} < 1$.
\end{description}
In either case, $|G| = |G_{1}| + |G_{R}| + |G_{B}| < |G|$, a contradiction.



Finally we consider Subcase~\ref{SubCase4.2}. Let $G_{1}$ be a part of the Gallai partition containing at least one red edge (but no red triangle). Again partition the remaining vertices of $G$ into $G_{R}$ and $G_{B}$ such that $G_{R}$ contains all vertices in parts having red edges to $G_{1}$ and $G_{B}$ contains all vertices in parts having blue edges to $G_{1}$.

Then $G_{R}$ contains no red $K_{3}$ and no blue $K_{3}$. Similarly, $G_{B}$ contains no red $K_{5}$ and no blue edges at all. This means that
$$
|G_{R}| \leq g(r - 1, s, t + 1) ~ ~ ~ ~ ~ \text{and} ~ ~ ~ ~ ~ |G_{B}| \leq g(r, s, t - 1).
$$
Furthermore, since $G_{1}$ contains red edges but no red triangle and no blue edges, we get
$$
|G_{1}| \leq g(r - 1, s, t).
$$
By considering each of the cases of the statement ($(c_{1})$ up to $(c_{11})$) individually across Types T1, T11, and T12 in Tables~\ref{Table:1} and~\ref{Table:2}, we see that in each of the cases, we have
$$
\frac{g(r - 1, s, t + 1) + g(r - 1, s, t) + g(r, s, t - 1)}{g(r, s, t)} \leq 1
$$
except in the two cases $(c_{7})$ and $(c_{9})$.

Certainly $G_{B}$ contains no blue edges so every pair of parts in $G_{B}$ are joined by red edges. Since $G$ has no red $K_{5}$, there must be at most $4$ parts in $G_{B}$. By Fact~\ref{Fact:Deg1}, no individual part contains a red $K_{4}$ so if there is only one part in $G_{B}$, it has order at most $g(r - 1, s + 1, t - 1)$. If a part $G_{i} \subseteq G_{B}$ contains a red triangle, then there can only be one other part, which must have no red edge. In this case $|G_{B}| \leq g(r - 1, s + 1, t - 1) + g(r - 1, s, t - 1)$.  If two parts have red edges (but no red triangle), then these are the only two parts in $G_{B}$ and $|G_{B}| \leq 2g(r - 1, s, t)$. Finally if one part has red edges (but no red triangle), then there are at most $3$ parts in $G_{B}$, meaning that $|G_{B}| \leq g(r - 1, s, t) + 2g(r - 1, s, t - 1)$.
Putting these observations together, we see that
\beqs
|G_{B}| & \leq & \begin{cases}
g(r - 1, s + 1, t - 1),\\
g(r - 1, s + 1, t - 1) + g(r - 1, s, t - 1),\\
2g(r - 1, s, t),\\
g(r - 1, s, t) + 2g(r - 1, s, t - 1)
\end{cases}\\
~ & \leq & g(r - 1, s + 1, t - 1) + g(r - 1, s, t - 1).
\eeqs
This means that
\beqs
|G| & = & |G_{1}| + |G_{R}| + |G_{B}|\\
~ & \leq & g(r - 1, s, t) + g(r - 1, s, t + 1)\\
~ & ~ & + g(r - 1, s + 1, t - 1) + g(r - 1, s, t - 1)\\
~ & \leq & g(r, s, t)\\
~ & < & |G|,
\eeqs
a contradiction, completing the proof of Case~\ref{Case:4}.

\begin{case}\labelz{Case:5}
Red is among the first $r$ colors while blue is among the middle $s$ colors.
\end{case}

In this case, the graph $G$ contains no red $K_{5}$ and no blue $K_{4}$. Since $r(K_{4}, K_{5}) = 25$, we find that $4 \leq q \leq 24$. We break the proof into subcases based on the red and blue edges that appear within parts of a Gallai partition. These subcases are listed as follows.
\bd
\item{\ref{Subcase:5.1}.} No part of the partition contains any red or blue edges.
\item{\ref{Subcase:5.2}.} A part $G_{1}$ contains a red copy of $K_{3}$ and at least one blue edge. 
\item{\ref{Subcase:5.3}.} A part $G_{1}$ contains red and blue edges, but no red or blue copy of $K_{3}$.
\item{\ref{Subcase:5.4}.} A part $G_{1}$ contains a red copy of $K_{3}$ and no blue edges.
\item{\ref{Subcase:5.5}.} A part $G_{1}$ contains blue edges.
\item{\ref{Subcase:5.6}.} A part $G_{1}$ contains red edges but no red copy of $K_{3}$, and no blue edges.
\ed

We now consider these subcases.

\begin{subcase}\labelz{Subcase:5.1}
No part of the partition contains any red or blue edges.
\end{subcase}


If each part $G_{i}$ of the partition contains no red or blue edges, then this is Type T16 so by Inequality~\eqref{Ineq:T16}, we get
$\frac{g(r - 1, s - 1, t)}{g(r, s, t)} \leq \frac{1}{24}$. This means that
\beqs
|G| & = & \sum_{i = 1}^{q} |G_{i}|\\
~ & \leq & \sum_{i = 1}^{q} g(r - 1, s - 1, t)\\
~ & \leq & \sum_{i = 1}^{q} \frac{1}{24} g(r, s, t)\\
~ & = & \frac{q}{24} g(r, s, t)\\
~ & \leq & g(r, s, t),
\eeqs
a contradiction, completing the proof for Subcase~\ref{Subcase:5.1}.



\begin{subcase}\labelz{Subcase:5.2}
A part $G_{1}$ contains a red copy of $K_{3}$ and at least one blue edge.
\end{subcase}

Let $G_{R}$ (and $G_{B}$) be the set of vertices with all red (respectively blue) edges to $G_{1}$. Then $|G_{B}|$ can be bounded from above by $g(r, s - 1, t)$. 
By Inequality~\eqref{Ineq:T4}, we get
$$
\frac{g(r, s - 1, t)}{g(r, s, t)} \leq \frac{1}{3}.
$$
We can also bound $|G_{1}|$ from above by $g(r - 1, s, t + 1)$. By Inequality~\eqref{Ineq:T11}, we get
$$
\frac{g(r-1, s, t+1)}{g(r, s, t)} \leq \frac{17}{36}.
$$

Finally, since $G_R$ contains no red edges, it has at most three parts. No part has a blue $K_3$ so there can be
three parts with no blue edges or one part with blue edges and another with no blue edges. Thus, we obtain
$$
\frac{|G_{R}|}{g(r, s, t)} \leq \max \left\{3 \cdot \frac{1}{24}, \frac{1}{9} + \frac{1}{24}\right\} = \frac{11}{72}.
$$
Summarizing, we get
$$
|G|=|G_1|+|G_R|+|G_B| \leq \left(\frac{17}{36}+\frac{11}{72}+\frac{1}{3}\right)g(r,s,t) = \frac{69}{72}g(r,s,t) < g(r,s,t),
$$
a contradiction.

\begin{subcase}\labelz{Subcase:5.3}
A part $G_{1}$ contains red and blue edges, but no red or blue copy of $K_{3}$.
\end{subcase}

Then $G_r$ is of Type T11 so by Inequality~\eqref{Ineq:T11}, 
we obtain
$$
\frac{g(r-1, s, t+1)}{g(r, s, t)} \leq \frac{17}{36}.
$$
Similarly, $G_1$ is of Type T12 so by Inequality~\eqref{Ineq:T12}, 
we have
$$
\frac{g(r-1, s-1, t+2)}{g(r, s, t)} \leq \frac{2}{9}.
$$

Since $G_B$ has no blue edges, it has at most four parts. There can be four parts with no red edges, or one part with red edges and two parts with no red edges, or two parts with a red $K_3$ in one part and no red edges in the other part, or two parts with red edges, or one part with a red $K_3$. This means that
$$
\frac{|G_{B}|}{g(r, s, t)} \leq \max \left\{4 \cdot \frac{1}{24}, \frac{1}{9} + \frac{2}{24}, \frac{17}{72}+\frac{1}{24}, 2 \cdot \frac{1}{9}, \frac{17}{72}\right\}g(r,s,t) = \frac{5}{18}g(r,s,t).
$$

Summarizing we obtain
$$
|G|=|G_1|+|G_R|+|G_B| \leq \left(\frac{2}{9}+\frac{17}{36}+\frac{5}{18}\right)g(r,s,t) = \frac{35}{36}g(r,s,t) < g(r,s,t),
$$
a contradiction.




\begin{subcase}\labelz{Subcase:5.4}
A part $G_{1}$ contains a red copy of $K_{3}$ and no blue edges.
\end{subcase}

First a claim that there is only one such part.

\begin{claim}
At most one part contains a red copy of $K_3.$
\end{claim}

\begin{proof}
Suppose there are two parts $G_1$ and $G_2$ containing a red copy of $K_3$. To avoid a red copy of $K_{5}$, there must be all blue edges in between $G_{1}$ and $G_{2}$. Then $G_R$, the set of vertices with red edges to $G_{1}$, is an $(R_2,B_4)$-graph, and $G_{2}$ is contained in $G_{B}$, the set of vertices with all blue edges to $G_{1}$. We also see that the set of vertices in $G_{B}$ with all red edges to $G_{1}$, $G_{BR}$ is an $(R_2,B_3)$-graph, and the set of vertices in $G_{B}$ with all blue edges to $G_{2}$, $G_{BB}$ is an $(R_5,B_2)$-graph. We deduce that
$$
w(G_R) \leq \max \left\{2 \cdot \frac{1}{24}, \frac{1}{9} + \frac{1}{24}\right\} = \frac{11}{72},
$$

$$
w(G_{BR}) \leq \max \left\{\frac{1}{9}, 2 \cdot \frac{1}{24}\right\} = \frac{1}{9},
$$

\beqs
w(G_{BB}) & \leq & \frac{|G_{B}|}{g(r, s, t)} \\
~ & \leq & \max \left\{4 \cdot \frac{1}{24}, \frac{1}{9} + \frac{2}{24}, \frac{17}{72}+\frac{1}{24}, 2 \cdot \frac{1}{9}, \frac{17}{72}\right\}g(r,s,t) \\
~ & = & \frac{5}{18}g(r,s,t).
\eeqs

We now distinguish two cases. First if $G_R$ contains a blue part, then all edges from this blue part to $G_{BB}$ are red. So $G_{BB}$ contains no red copy of $K_4$. Then
$$
w(G_{BB}) \leq \max \left\{\frac{17}{72}, \frac{1}{9} + \frac{1}{24}, 3 \cdot \frac{1}{24}\right\} = \frac{17}{72}.
$$
Thus, we obtain 
\beqs
|G| & = & |G_1|+|G_2|+|G_R|+|G_{BR}|+|G_{BB}|\\
~ & \leq & \left(2 \cdot \frac{5}{26}+\frac{17}{72}+\frac{1}{9}+\frac{17}{72}\right)g(r,s,t)\\
~ & = & \frac{151}{156}g(r,s,t)\\
~ & < & g(r,s,t),
\eeqs
a contradiction.

If $G_{BR}$ contains no blue part, then $w(G_{BR}) \leq \frac{2}{24}$, so we obtain
\beqs
|G| & = & |G_1|+|G_2|+|G_R|+|G_{BR}|+|G_{BB}|\\
~ & \leq & \left(2 \cdot \frac{17}{72}+\frac{11}{72}+\frac{2}{24}+\frac{5}{18}\right)g(r,s,t)\\
~ & = & \frac{71}{72}g(r,s,t)\\
~ & < & g(r,s,t),
\eeqs
a contradiction.
\end{proof}

Let $G_1$ be a part containing a red copy of $K_3.$ Then $G_R$ is an $(R_2,B_4)$-graph and $G_B$ is an $(R_5,B_3)$-graph. We deduce that
$$
w(G_{R}) \leq \max \left\{\frac{1}{9}+\frac{1}{24}, 3 \cdot \frac{1}{24}\right\} = \frac{11}{72},
$$
and $w(G_B) \leq \frac{5}{9}$ by Lemma~\ref{Lem:5.3}. Thus, we obtain
$$
|G|=|G_1|+|G_R|+|G_B| \leq \left(\frac{17}{72}+\frac{11}{72}+\frac{5}{9}\right)g(r,s,t) = \frac{17}{18}g(r,s,t) < g(r,s,t),
$$
a contradiction.

For the remaining subcases, we therefore know that each part $G_{i}$ can contain red (or blue) edges but no red (respectively blue) copy of $K_{3}$, and no blue (respectively red) edges.

\begin{subcase}\labelz{Subcase:5.5} 
A part $G_{1}$ contains a blue edges.
\end{subcase}

Suppose $G_{1}$ contains a blue edge. Let $G_{R}$ (or $G_{B}$) denote the set of vertices in $G \setminus G_{1}$ with red (respectively blue) edges to $G_{1}$. Let $q_{1}$ be the number of parts of the Gallai partition in $G_{R}$ and let $q_{2}$ be the number of parts of the Gallai partition in $G_{B}$. Then
$$
q_{1} \leq 17 = R(K_{4}, K_{4}) - 1
$$
since the reduced graph of $G_{R}$ contains no monochromatic copy of $K_{4}$. Similarly, $q_{2} \leq 4$ since $G_{B}$ contains no blue edges.

For the situation when $G_{R}$ contains either red or blue edges, we apply Inequality~\eqref{Ineq:T15} to get 
$|G_{R}| \leq g(r - 1, s - 1, t + 1) \leq \frac{1}{9}g(r, s, t)$.


We intend to show that
$$
|G| = |G_{1}| + |G_{R}| + |G_{B}| \leq g(r, s, t),
$$
which would be a contradiction. This would hold if we could show that
$$
|G_{1}| + |G_{B}| \leq \frac{1}{3} g(r, s, t) ~ ~ ~ \text{ and } ~ ~ ~ |G_{R}| \leq \frac{2}{3} g(r, s, t)
$$
or equivalently if
$$
\frac{|G_{1}| + |G_{B}|}{g(r, s, t)} \leq \frac{1}{3} ~ ~ ~ \text{ and } ~ ~ ~ \frac{|G_{R}|}{g(r, s, t)} \leq \frac{2}{3}.
$$

Certainly there are no blue edges within $G_{B}$ so every pair of parts in $G_{B}$ is joined entirely by red edges. Since $G$ contains no red copy of $K_{5}$, there can be at most $4$ parts of the Gallai partition in $G_{B}$.
If two parts have red edges (but certainly no red triangle in this subcase), then these are the only two parts in $G_{B}$ and if one part has red edges (but no red triangle), then there are at most $3$ parts in $G_{B}$.

This means that
$$
\frac{|G_{1}| + |G_{B}|}{g(r, s, t)} \leq \max \left\{\frac{1}{9} + 2 \cdot \frac{1}{9}, \frac{1}{9} + \frac{1}{9} + 2\cdot \frac{1}{24}, \frac{1}{9} + 4 \cdot \frac{1}{24} \right\} = \frac{3}{9} = \frac{1}{3},
$$
as desired.

Within $G_{R}$, we first note that there is no red copy of $K_{4}$ and no blue copy of $K_{4}$. We may therefore follow along with the proof of Case~\ref{Case:3} with the following arguments, cases concerning the possible values of $q_{1}$.


First suppose $q_{1} = 17$. If a part in $G_{2}$ in $G_{B}$ contains a red edge, then since $R(K_{3}, K_{4}) = 9$, there must be at most $8$ parts in $G_{R}$ with red edges to $G_{2}$ and at most $8$ parts in $G_{R}$ with blue edgs to $G_{2}$, meaning $q_{1} \leq 16$. Thus, no part in $G_{B}$ has a red edge, so
$$
\frac{|G_{R}|}{g(r, s, t)} + \frac{|G_{1}| + |G_{B}|}{g(r, s, t)} \leq \frac{q_{1}}{24} + \frac{4}{24} \leq \frac{21}{24} < 1,
$$
a contradiction. If $13 \leq q_{1} \leq 16$, then
$$
\frac{|G_{R}|}{g(r, s, t)} \leq \frac{16}{24} = \frac{2}{3},
$$
as claimed.


When $4 \leq q_{1} \leq 12$, we apply Claim~\ref{Claim:VbVr3} and note that $p_{2} = 0$. Following the proof of Subcase~\ref{Subcase:3.2}, we get the following. If $q_{1} \leq 8$, then
$$
p_{1} \cdot \frac{1}{9} + p_{0} \cdot \frac{1}{24} \leq \frac{4}{9} + \frac{4}{24} < \frac{2}{3}.
$$
If $q_{1} = 9$, then
$$
p_{1} \cdot \frac{1}{9} + p_{0} \frac{1}{24} \leq \frac{3}{9} + \frac{6}{24} = \frac{1}{3} + \frac{1}{4} < \frac{2}{3}.
$$
If $q_{1} = 10$, then $p_{1} \leq 2$ and
$$
p_{1} \cdot \frac{1}{9} + p_{0} \cdot \frac{1}{24} \leq \frac{2}{9} + \frac{8}{24} < \frac{2}{3}.
$$
Finally if $11 \leq q_{1} \leq 12$, then $p_{1} \leq 2$ and
$$
p_{1} \cdot \frac{1}{9} + p_{0} \cdot \frac{1}{24} \leq \frac{2}{9} + \frac{10}{24} = \frac{23}{36} < \frac{2}{3},
$$
completing the proof of Subcase~\ref{Subcase:5.5}.

For the remaining subcase, all parts are either free or red, and there is at least one red part.

\begin{subcase}\labelz{Subcase:5.6}
A part $G_{1}$ contains red edges but no red copy of $K_{3}$, and no blue edges.
\end{subcase}

Then $G_R$ is an $(R_3,B_4)$-graph and $G_B$ is an $(R_5,B_3)$-graph. So by Inequality~\eqref{Ineq:T15} and Lemmas~\ref{Lem:5.2} and~\ref{Lem:5.3}, we obtain
\beqs
|G|& = & |G_1|+|G_R|+|G_B|\\
~ & \leq & \max \left\{ \frac{1}{9}+\frac{1}{3}+\frac{5}{9}, \frac{1}{9} + \frac{25}{72} + \frac{39}{72} \right\} \cdot g(r,s,t) \\
~ & = & g(r,s,t) < g(r,s,t)+1,
\eeqs
a contradiction, unless $G_{R}$ and $G_{B}$ are both very specific blow-ups of the unique $2$-coloring of $K_{5}$ with no monochromatic triangle as in Lemmas~\ref{Lem:5.2} and~\ref{Lem:5.3}.

In order to avoid creating a red $K_{5}$, each part in $G_{R}$ can have red edges to at most two parts in $G_{B}$. This means that there must be at least $15$ pairs of parts (one in $G_{R}$ and one in $G_{B}$) with blue edges between them. To avoid creating a blue copy of $K_{4}$, the only way for a part in $G_{B}$ to have blue edges to three parts in $G_{R}$ is for all of those blue edges to go to the free parts in $G_{R}$. This leaves all red edges from the blue parts in $G_{R}$ to $G_{B}$, making a red copy of $K_{5}$, completing the proof of Subcase~\ref{Subcase:5.6} and Case~\ref{Case:5}.









\begin{case}\labelz{Case:6}
Red and blue are both within the first $r$ colors.
\end{case}

In this case, the graph $G$ contains no red or blue $K_{5}$. Since $r(K_{5}, K_{5}) = R + 1$, we find that $4 \leq q \leq R$. It turns out that a better bound is almost immediate.

\begin{claim}\labelz{Claim:Bigq}
If red and blue occur within the first $r$ colors, then
$$
q \leq 38.
$$
\end{claim}

\begin{proof}
Suppose, for a contradiction, that $q \geq 39$. By Fact~\ref{Fact:LargeDegrees}, for each vertex $w_{i} \in D$, we have $d_{r}(w_{i}) \leq 24$ and $d_{b}(w_{i}) \leq 24$. Since $q = |D| \geq 39$, this means that $d_{r}(w_{i}) \geq 14$ and $d_{b}(w_{i}) \geq 14$. By Fact~\ref{Fact:SmallDegrees}, we get that $V_{r} = \emptyset$ and $V_{b} = \emptyset$, so $p_{0} = q$. By induction on the number of colors,
$$
|G_{i}| \leq \frac{1}{R} g(r, s, t)
$$
so this means that
$$
|G| \leq \frac{q}{R} g(r, s, t) < |G|,
$$
a contradiction.
\end{proof}


We break the remainder of the proof of this case into the following subcases:
\be
\item There is a part $G_{1}$ containing a red triangle and a blue triangle but no red or blue copy of $K_{4}$. 
\item There is a part $G_{1}$ containing a red edge and a blue triangle but no red triangle and no blue copy of $K_{4}$. 
\item There is a part $G_{1}$ containing a red edge and a blue edge but no red or blue triangle. So $G_1$ is an $(R_{3}, B_{3})$-part.
\item There is an $(R_{2}, B_{4})$-part $G_{1}$.
\item Each part is either a free part or a red part or a blue part.
\ee

\begin{subcase} 
There is a part $G_{1}$ containing a red triangle and a blue triangle but no red or blue copy of $K_{4}$.
\end{subcase}

If we let $G_{R}$ and $G_{B}$ be the sets of vertices with all red or respectively blue edges to $G_{1}$, then it is clear that $G_{R}$ contains no red edges and $G_{B}$ contains no blue edges. Since all edges between parts in $G_{R}$ must be blue, there can be at most $4$ parts and similar there can be at most $4$ parts in $G_{B}$. Since $G_{1}$ contains a red triangle and a blue triangle but no red or blue copy of $K_{4}$, we see from Inequality~\eqref{Ineq:T17} that
$$
\frac{|G_{1}|}{|G|} \leq \frac{18}{R}.
$$
The orders of $G_{R}$ and $G_{B}$ satisfy identical bounds so, by symmetry, we will consider only $|G_{R}|$.

If $G_{R}$ contains only one part, this part contains no red edges and perhaps some blue triangles but no blue copy of $K_{4}$ (recall that for every part $G_{i}$, there exists a part $G_{j}$ with all blue edges to $G_{i}$). By Inequality~\eqref{Ineq:T19}, this means that $|G_{R}|/|G| \leq \frac{6}{R}$.

If $G_{R}$ contains two parts, these must have all blue edges between them. Then either one of these parts contains a blue triangle and the other contains no blue edges, or each part contains blue edges but no blue triangle. In the former situation, by Inequalities~\eqref{Ineq:T22} and~\eqref{Ineq:T19}, we have
$$
\frac{|G_{R}|}{|G|} \leq \frac{6}{R} + \frac{1}{R} = \frac{7}{R}.
$$
In the latter situation, by Inequality~\eqref{Ineq:T20}, we have
$$
\frac{|G_{R}|}{|G|} \leq 2 \frac{13}{2\cdot R} = \frac{13}{R}.
$$

If $G_{R}$ contains three parts, at most one of them can contain any blue edges so, by Inequalities~\eqref{Ineq:T22} and~\eqref{Ineq:T21}, we have
$$
\frac{|G_{R}|}{|G|} \leq \frac{13}{4 \cdot R} + 2 \frac{1}{R} = \frac{21}{4\cdot R}.
$$
Finally if $G_{R}$ contains four parts, none of these may contain any blue edges so, by Inequality~\eqref{Ineq:T22}, we have
$$
\frac{|G_{R}|}{|G|} \leq 4 \frac{1}{R} = \frac{4}{R}.
$$
Putting these together, we have $\frac{|G_{R}|}{|G|} \leq \frac{7}{R}$ and so symmetrically we also get $\frac{|G_{B}|}{|G|} \leq \frac{7}{R}$. These imply that
$$
|G| = |G_{1}| + |G_{R}| + |G_{B}| \leq \frac{|G|}{R}( 18 + 7 + 7) < |G|,
$$
a contradiction, completing the proof of this subcase.



\begin{subcase} 
There is a part $G_{1}$ containing a red edge and a blue triangle but no red triangle and no blue copy of $K_{4}$.
\end{subcase}

Again let $G_{R}$ and $G_{B}$ be the sets of vertices with all red or respectively blue edges to $G_{1}$, so $G_{R}$ contains no red triangle and $G_{B}$ contains no blue edges. By Inequality~\eqref{Ineq:T18}, we see that $|G_{1}| \leq \frac{12}{R} |G|$. From the same argument as in the previous subcase, we see that $|G_{B}| \leq \frac{7}{R} |G|$. By Inequality~\eqref{Ineq:T11}, we see that $|G_{R}| \leq \frac{5}{13} |G|$.
Putting all these together, we get
$$
|G| = |G_{1}| + |G_{B}| + |G_{R}| \leq \left(\frac{12}{R} + \frac{7}{R} + \frac{5}{13}\right)|G| < |G|,
$$
a contradiction, completing the proof of this subcase.



\begin{subcase}
There is an $(R_{3}, B_{3})$-part $G_{1}$.
\end{subcase}

With $G_{R}$ and $G_{B}$ being the sets of vertices with red or blue edges respectively to $G_{1}$, we consider several possible situations. We further break into cases based on the surrounding structures.

\begin{subsubcase}
No other part contains red or blue edges.
\end{subsubcase}

Then since $G_{B}$ contains no blue triangle and no red copy of $K_{5}$, we see that $G_{B}$ contains at most $R(K_{3}, K_{5}) - 1 = 13$ parts of the Gallai partition and similarly $G_{R}$ also contains at most $13$ parts of the Gallai partition. This means that
$$
|G| = |G_{1}| + |G_{B}| + |G_{R}| \leq \frac{|G|}{R} \left( \frac{13}{2} + 13 + 13\right) = \frac{32.5|G|}{R} < |G|,
$$
a contradiction.

\begin{subsubcase}
There is a $(R_{2}, B_{4})$-part $G_{2}$.
\end{subsubcase}

In order to avoid creating a blue copy of $K_{5}$, all edges from $G_{1}$ to $G_{2}$ must be red. Let $F_{2B}$ denote the set of vertices with blue edges to $G_{2}$, let $F_{1R}$ denote any remaining vertices with red edges to $G_{1}$ and let $F_{1B}$ denote the set of vertices with blue edges to $G_{1}$. Note that $F_{2B}$ contains no blue edges, $F_{1R}$ contains no red edges since both $G_{1}$ and $G_{2}$ have all red edges to $F_{1R}$, and $F_{1B}$ contains no blue triangle and no red copy of $K_{4}$. This means that
\beqs
|G| & = & |G_{1}| + |G_{2}| + |F_{1B}| + |F_{1R}| + |F_{2B}|\\
~ & \leq & \frac{|G|}{R} \left( \frac{13}{2} + 6 + 12 + 7 + 7 \right)\\
~ & = & \frac{38.5|G|}{R} < |G|,
\eeqs
a contradiction.

\begin{subsubcase}
There is another $(R_{3}, B_{3})$-part $G_{2}$.
\end{subsubcase}

Without loss of generality, suppose the edges between $G_{1}$ and $G_{2}$ are red. If we let $F_{1R}$ denote the set of vertices with red edges to $G_{1}$, then all of $F_{1R}$ must have blue edges to $G_{2}$. Let $F_{1B}$ be the remaining vertices, those with blue edges to $G_{1}$. Then $F_{1B}$ contains no blue triangle and $F_{1R}$ contains no blue or red triangle. This means that
\beqs
|G| & = & |G_{1}| + |G_{2}| + |F_{1R}| + |F_{1B}|\\
~ & \leq & \frac{|G|}{R} \left( \frac{13}{2} + \frac{13}{2} + \frac{13}{2}\right) + \frac{5}{13}|G|\\
~ & \leq & \left( \frac{18.5}{R} + \frac{5}{13} \right) |G| < |G|,
\eeqs
a contradiction.


\begin{subsubcase} 
Every other part is either free, red or blue.
\end{subsubcase}

Note that $G_R$ is an $(R_3,B_5)$-graph and $G_B$ is an $(R_5,B_3)$-graph. So by Lemma~\ref{Lem:6.3}, we obtain
$$
|G| = |G_{1}| + |G_{R}| + |G_{B}| \leq \left( \frac{13}{2R}|G| + 2 \cdot \frac{16.25}{R}\right) g(r,s,t) = \frac{39}{R} g(r,s,t) < g(r,s,t),
$$
a contradiction, completing the proof of this subsubcase.

\begin{subcase} 
There is an $(R_{2}, B_{4})$-part $G_{1}$. 
\end{subcase}

With $G_{R}$ and $G_{B}$ being the sets of vertices with red or blue edges respectively to $G_{1}$, the graph induced on $G_{R}$ contains no blue edge. We consider several possible situations and further break into cases based on the surrounding structures.

\begin{subsubcase}
No part in $G_{R}$ contains red or blue edges.
\end{subsubcase}

Then since $G_{B}$ contains no blue edges and no red copy of $K_{5}$, we see that $G_{B}$ contains at most $R(K_{2}, K_{5}) - 1 = 4$ parts of the Gallai partition and similarly $G_{R}$ also contains at most $R(K_{4}, K_{5}) - 1 = 24$ parts of the Gallai partition. This means that
$$
|G| = |X_{1}| + |G_{B}| + |G_{R}| \leq \frac{|G|}{R} \left( 6 + 4 + 24 \right) = \frac{34|G|}{R} < |G|,
$$
a contradiction.

\begin{subsubcase}\labelz{Case:6.4.1}
There is an $(R_{4}, B_{2})$-part $G_{2}$ in $G_{R}$.
\end{subsubcase}

Without loss of generality, suppose the edges between $G_{1}$ and $G_{2}$ are all red. Let $F_{1R}$ be the set of vertices (other than $G_{2}$) with red edges to $G_{1}$ and let $F_{1B}$ be the set of vertices with blue edges to $G_{1}$. Then $F_{1B}$ contains no blue edges and $F_{1R}$ must have blue edges to $G_{2}$ to avoid creating a red copy of $K_{5}$ so $F_{1R}$ contains no red or blue copy of $K_{4}$. This means that
\beqs
|G| & = & |G_{1}| + |G_{2}| + |F_{1B}| + |F_{1R}|\\
~ & \leq & \frac{|G|}{R} (6 + 6 + 7 + 18)\\
~ & = & \frac{37}{R} |G| < |G|,
\eeqs
a contradiction.

\begin{subsubcase}
There is an $(R_{3}, B_{2})$-part $G_{2}$ in $G_{R}$.
\end{subsubcase}

Let $G_{RR}$ be the set of vertices in $G_{R}$ with red edges to $G_{2}$ and let $G_{RB}$ be those vertices in $G_{R}$ with blue edges to $G_{2}$. Then $G_{B}$ contains no blue edges, $G_{RR}$ contains no red edges, and $G_{RB}$ contains no red or blue copy of $K_{4}$. This means that
\beqs
|G| & = & |G_{1}| + |G_{2}| + |G_{B}| + |G_{RR}| + |G_{RB}|\\
~ & \leq & \frac{|G|}{R} \left(6 + \frac{13}{4} + 7 + 7 + 18\right)\\
~ & = & \frac{41.25}{R} |G| < |G|,
\eeqs
a contradiction.


\begin{subsubcase}
There is an $(R_{2}, B_{3})$-part $G_{2}$ in $G_{R}$.
\end{subsubcase}

Note that we may assume that $G_{B}$ contains no red triangle since if it did, this structure would be symmetric to the assumed structure considered in Subcase~\ref{Case:6.4.1}.

Let $G_{RR}$ denote the set of vertices in $G_{R}$ with red edges to $G_{2}$ and let $G_{RB}$ denote the set of vertices in $G_{R}$ with blue edges to $G_{2}$. Then $G_{B}$ contains no blue edges and $G_{RB}$ contains no red $K_{4}$ and no blue triangle.

Hence $G_{RR}$ is an $(R_3,B_5)$-graph and $G_{RB}$ is a $(R4,B3)$-graph. Using Lemmas~\ref{Lem:6.2} and~\ref{Lem:6.3}, we obtain
\beqs
|G| & = & |G_{1}| + |G_{2}| + |G_{B}| + |G_{RR}| + |G_{RB}|\\
~ & \leq & \frac{|G|}{R} \left( 6 + \frac{13}{4} + 7 + 16.25 + 9.75\right)\\
~ & = & \frac{42.25}{R}|G| < |G|,
\eeqs
a contradiction.

\begin{subsubcase}\labelz{Case:6.4.5}
There is an $(R_{2}, B_{4})$-part $G_{2}.$
\end{subsubcase}

Then all edges between $G_{1}$ and $G_{2}$ are red. Let $F_{1B}$ be the set of vertices with blue edges to $G_{1},$ let $F_{RR}$ be the set of vertices with red edges to both
$G_{1}$  and $G_2,$ and let $F_{RB}$ be the set of vertices with red edges to $G_{1}$ and blue edges to$G_2.$ Then $F_{1B}$ contains no blue edges and
$F_{RB}$ contains no blue edges and no red $K_4.$ If $F_{RR}$ contains no blue $K_3,$ then $F_{RR}$ is an $(R_3,B_5)$-graph.
This means that
\beqs
|G| & = & |G_{1}| + |G_{2}| + |F_{1B}| + |F_{1R}| + |F_{RR}|\\
~ & \leq & \frac{|G|}{R} (6 + 6 + 7 + 16.25) + 6\\
~ & = & \frac{41.25}{R} |G| < |G|,
\eeqs
a contradiction.

Suppose there is an $(R_{2}, B_{4})$-part $G_{3}$ in $F_{RR}.$ Repeating above arguments leads to
\beqs
|G| & = & |G_{1}| + |G_{2}| + |G_3| + |F_{1B}| + |F_{1R}| + |F_{3R}| + |F_{3B}|\\
~ & \leq & \frac{|G|}{R} (4 \cdot 6 + 2 \cdot + 7 + 3.25) \\
~ & = & \frac{41.25}{R} |G| < |G|,
\eeqs
a contradiction.







\begin{subcase} 
Each part is either a free part or a red part or a blue part.
\end{subcase}

First suppose that $G$ has exactly one non-free part. Let $G_1$ be this part containing red edges.
Then $G_R$ is an $(R_3,B_5)$-graph and $G_B$ is an $(R_5,B_4)$-graph. So we obtain

$$
|G| \leq \frac{1}{R}\left(\frac{13}{4} + 13 + 24\right) = \frac{40.25}{R} \leq \frac{40.25}{43} < 1,
$$
a contradiction.

Hence we may assume that $G$ contains at least two non-free parts, say $G_1$ and $G_2$. If
$G_1$ and $G_2$ both contain red edges, and $G_1$ and $G_2$ are joined by blue edges, then we call this
a RBR-pair. Analogously, RRR-pairs, RRB-pairs (BBR-pairs), BBB-pairs, BRB-pairs, and BBR-pairs (RBB)-pairs are defined.

\begin{claim}
$G$ contains an RRR-pair or a BBB-pair.
\end{claim}

\begin{proof}
Suppose not. First assume that $G$ contains a RBR-pair. So let $G_1$ and $G_2$ contain red edges and $G_1$ and $G_2$
are joined by blue edges. Then $G_R$ is an $(R_3,B_5)$-graph, $G_{BB}$ is an $(R_5,B_3)$-graph and $G_{BR}$ is an $(R_3,B_4)$-graph.
Since there is no RRR-pair, both $G_R$ and $G_{BR}$ contain no red parts. Hence $w(G_{BR}) \leq \frac{9.5}{R}$ by Lemma~\ref{Lem:6.2} (ii).
Now $G_R$ can have at most two blue parts, since otherwise there is a BBB-pair. Now by Lemma~\ref{Lem:6.3}, we conclude that
$w(G_{R}) \leq \frac{13.5}{R}.$ Using the same arguments, we conclude that $w(G_{BB}) \leq \frac{13.5}{R}.$ Thus, we obtain
$$
|G| \leq \frac{1}{R}\left(2 \cdot \frac{13}{4} + 2 \cdot 13.5  + 9.5\right)g(r,s,t) = \frac{43}{R}g(r,s,t) \leq g(r,s,t),
$$
a contradiction.

Hence we may assume that $G$ contains no RBR-pair, no BRB-pair, but a RRB-pair (BBR-pair). Let $G_1$ contain red edges, let $G_2$ contain blue edges and $G_1$ and $G_2$
are joined by red edges. Then $G_{RR}$ is an $(R_2,B_5)$-graph, $G_{RB}$ is an $(R_3,B_3)$-graph and $G_{B}$ is an $(R_5,B_4)$-graph.

By the assumptions there are no red parts in $G_{RR}, G_{RB}$ and $G_B,$ and no blue part in $G_{RB}.$ Furthermore, $G_{RR}$ and $G_B$
have at most one blue part by the assumption. So we conclude that

$$
w(G_{RR}) \leq \max \left\{4 \cdot \frac{1}{R}, \frac{13}{4R} + 2 \cdot \frac{1}{R}, 3 \cdot \frac{1}{R}\right\} = \frac{5.25}{R},
$$
$$
w(G_{RB}) \leq (R(3,3)-1)) \cdot \frac{1}{R} = \frac{5}{R},
$$
$$
w(G_{B}) \leq \max \left\{24 \cdot \frac{1}{R}, (24-1) \cdot \frac{1}{R} + \frac{13}{4R}\right\} = \frac{26.25}{R}.
$$
Thus, we obtain
$$
|G| \leq \frac{1}{R}\left(2 \cdot \frac{13}{4} + 5.25 + 5   + 26.25\right)g(r,s,t) = \frac{43}{R}g(r,s,t) \leq g(r,s,t),
$$
a contradiction.
\end{proof}

We now consider a RRR-pair, say $G_1$ and $G_2$ each red parts and joined by red edges. Then $G_R$ is an $(R_3,B_4)$-graph. Suppose first that there is no other red part in $G_B.$ If $G_B$ contains no blue part, then
$$
|G| = |G_1| + |G_2| + |G_{R}| + |G_{B}| \leq \frac{|G|}{R} \left(3.25 + 3.25 + 9.75 + 24 \right) = \frac{40.25|G|}{R} < |G|,
$$
a contradiction.

Suppose next that there is blue part $G_3$ in $G_B.$ Let $F_2$ be the set of parts which are joined by blue edges with
$G_2$ and $G_3,$ and let $F_3$ be the set of parts which are joined by blue edges with
$G_2$ and by red edges with $G_3.$ Then $F_2$ is an $(R_5,B_2)$-graph and $F_3$ is an $(R_4,B_4)$-graph.
Suppose $F_3$ contains no blue part, then
\beqs
|G| & = & |G_1| + |G_2| + |G_{R}| + |G_3| + |F_2| + |F_3|\\
~ & \leq & \frac{|G|}{R} \left(3.25 + 3.25 + 9.75 + 3.25 + 4 \cdot 1 + 17 \right)\\
~ & = & \frac{40.5|G|}{R} < |G|,
\eeqs
a contradiction. Now suppose that $F_3$ contains a blue part $G_4.$ Let $G_{4R}$ and $G_{4B}$
denote the parts joined by red or blue edges with $G_4.$ Then $G_{4R}$ is an $(R_3,B_4)$-graph and
$G_{4B}$ is an $(R_4,B_2)$-graph. So we obtain
\beqs
|G| & = & |G_1| + |G_2| + |G_{R}| + |G_3| + |F_2| + |G_4| + |G_{4R}| + |G_{4B}|\\
~ & \leq & \frac{|G|}{R} \left(3.25 + 3.25 + 9.75 + 3.25 + 4 \cdot 1 + 3.25 + 6.5 + 3 \cdot 1 \right)\\
~ & = & \frac{36.25|G|}{R} < |G|,
\eeqs
a contradiction.

Hence we may assume that there is a red part $G_3$ in $G_B.$ Let $F_1$ be the set of parts which are joined by red edges with $G_1$ and
by blue edges with $G_2,$ let $F_2$ be the set of parts which are joined by blue edges with
$G_2$ and $G_3,$ and let $F_3$ be the set of parts which are joined by blue edges with
$G_2$ and by red edges with $G_3.$ Then $F_1$ is an $(R_3,B_4)$-graph, $F_2$ is an $(R_5,B_2)$-graph, and $F_3$ is an $(R_4,B_4)$-graph.
By Lemma~\ref{Lem:6.2} we have $w(F_i) \leq \frac{9.75}{R}$ for $i=1,3.$ By Lemma~\ref{Lem:6.3}, if (i) or (ii) or (iii) holds, then
$w(F_2) \leq \frac{13.5}{R}$, so
\beqs
|G| & = & |G_1| + |G_2| + |G_3| + |F_1| + |F_2| + |F_3| \\
~ & \leq & \frac{|G|}{R} \left(3 \cdot 3.25 + 2 \cdot 9.75 + 13.5 \right)\\
~ & = & \frac{42.75|G|}{R} < |G|,
\eeqs
a contradiction.

Hence $F_2$ contains two red parts $G_3$ and $G_4$ joined by red edges.

Suppose first that $G_1$ and $G_4$ as well as $G_2$ and $G_3$ are joined by red edges.
Let $F_1$ be the set of parts which are joined by red edges with $G_1$ and
by blue edges with $G_2$ and $G_4,$ let $F_2$ be the set of parts which are joined by red edges with $G_3$ and by blue edges with
$G_2$ and $G_4,$ and let $F_3$ be the set of parts which are joined by blue edges with
$G_1$ and $G_3.$ Then $F_1$ and $F_2$ are $(R_3,B_3)$-graphs and $F_3$ is an $(R_5,B_3)$-graph.

By Lemmas~\ref{Lem:6.1} and~\ref{Lem:6.3}, we obtain
$$
|G| = \sum_{i=1}^4|G_i| + \sum_{j=1}^3|F_j| \leq \frac{|G|}{R} \left(4 \cdot 3.25 + 2 \cdot 6.5 + 16.25 \right) = \frac{42.25|G|}{R} < |G|,
$$
a contradiction.

Hence we may assume that $G_1$ is joined by blue edges with $G_3$ and $G_4.$ By Lemma~\ref{Lem:6.2}, we have $w(F_i) \leq \frac{9.75}{R}$ for $i=1,3.$ By Lemma~\ref{Lem:6.3}, if (i) or (ii) or (iii) holds, then
$w(F_2) \leq \frac{13.5}{R}.$
So we obtain
\beqs
|G| & = & |G_1| + |G_2| + |G_3| + |F_1| + |F_2| + |F_3|\\
~ & \leq & \frac{|G|}{R} \left(3 \cdot 3.25 + 2 \cdot 9.75 + 13.5 \right)\\
~ & = & \frac{42.75|G|}{R} < |G|,
\eeqs
a contradiction.

Hence $F_2$ contains two red parts joined by red edges. By symmetry, replacing $G_4$ by $G_3,$ we obtain by Lemma~\ref{Lem:6.3} (ii) that $G_3$ also contains
no blue parts.

Now we consider the subgraph $H$ spanned by $G_3, G_4, F_2$ and $F_3.$ These parts are all adjacent in blue to $G_1.$
If $w(H) \leq \frac{26.75}{R},$ then we obtain
$$
|G| = |G_1| + |G_2| + |F_1| + |H| \leq \frac{|G|}{R} \left(2 \cdot 3.25 + 9.75 + 26.75 \right) = \frac{43|G|}{R} \leq |G|,
$$
a contradiction.

Next we observe that $F_3$ is an $(R_3,B_3)$-graph. If $w(F_3) \leq 4,$ then we obtain
$$
w(H) \leq \frac{1}{R} \left(2 \cdot 3.25 + 5 \cdot 3.25 + 4 \right) = \frac{26.75}{R},
$$
which gives a contradiction as before.

So we may assume that $w(F_3) > 4.$ Now we obtain the following two final cases:

(i) $H$ contains nine red parts. \\
Since $R(3,4)=9,$ there is a blue $K_4$ or a red $K_3$ leading to a red $K_6,$ a contradiction.

(ii) $H$ contains eight red parts and a free part. Now contract every red part to a red vertex, we obtain
a graph $H'$ with eight red vertices and a vertex. Now $R(3,4)=9$ gives a blue $K_4$ or a red $K_3$ with at least two red vertices
implying that there is a red complete subgraph with at least $2 \cdot 2 + 1 = 5$ vertices, a contradiction, completing the proof of Case~\ref{Case:6} and Theorem~\ref{Thm:grK5}.
\end{proof}

\bibliography{../GR-Ref.bib}
\bibliographystyle{plain}

\end{document}

%% file: Tables.tex
Table~\ref{Table:1} contains the case analysis for the following inequalities:
\begin{eqnarray}
& & \text{Type T1: } \frac{g(r, s, t - 1)}{g(r, s, t)} \leq \frac{1}{2}, \label{Ineq:T1}\\
& & \text{Type T2: } \frac{g(r, s, t - 2)}{g(r, s, t)} \leq \frac{1}{5}, \label{Ineq:T2}\\
& & \text{Type T3: } \frac{g(r, s - 1, t + 1)}{g(r, s, t)} \leq \frac{2}{3}, \label{Ineq:T3}\\
& & \text{Type T4: } \frac{g(r, s - 1, t)}{g(r, s, t)}  \leq \frac{1}{3} , \label{Ineq:T4}\\
& & \text{Type T5: } \frac{g(r, s - 1, t - 1)}{g(r, s, t)} \leq \frac{1}{8}, \label{Ineq:T5}\\
& & \text{Type T6: } \frac{g(r, s - 2, t + 2)}{g(r, s, t)} \leq \frac{13}{36}. \label{Ineq:T6}
\end{eqnarray}

\begin{table}[H]
$$
\begin{array}{|c|c|c|c|c|c|c|l|}
\hline
\text{Case} & \text{\eqref{Ineq:T1}T1} & \text{\eqref{Ineq:T2}T2} & \text{\eqref{Ineq:T3}T3} & \text{\eqref{Ineq:T4}T4} & \text{\eqref{Ineq:T5}T5} & \text{\eqref{Ineq:T6}T6} & \\
\hline
(c_{1}) &  
\frac{2}{5} & 
\frac{1}{5} & 
\frac{8}{17} & 
\begin{array}{|c|}
\hline
\frac{3}{17}\\
\hline
\frac{16}{85}\\
\hline
\end{array} & 
\frac{8}{85} & 
\frac{5}{17} & 
\begin{array}{|c|}
\hline
t = 0\\
\hline
t \geq 2 \\
\hline
\end{array}  \\ 
\hline
(c_{2}) &  
\frac{1}{2} & 
\frac{1}{5} & 
\frac{8}{17} & 
\frac{4}{17} & 
\begin{array}{|c|}
\hline
\frac{3}{34}\\
\hline
\frac{8}{85}\\
\hline
\end{array} & 
\frac{5}{17} & 
\begin{array}{|c|}
\hline
t = 1 \\
\hline
t \geq 3\\
\hline
\end{array} \\ 
\hline
(c_{3}) &  
- & 
- & 
\frac{2}{15} & 
\frac{1}{3} & 
- & 
\frac{16}{51} & 
 \\ 
\hline
(c_{4}) &  
- & 
- & 
- & 
- & 
- & 
- & 
 \\ 
\hline
(c_{5}) &  
\begin{array}{|c|}
\hline
\frac{3}{8}\\
\hline
\frac{2}{5}\\
\hline
\end{array} & 
\frac{1}{5} & 
\frac{1}{8} & 
\frac{1}{4} & 
\frac{1}{8} & 
\frac{5}{17} & 
\begin{array}{|c|}
\hline
t = 1\\
\hline
t \geq 3 \\
\hline
\end{array} \\ 
\hline
(c_{6}) &  
\begin{array}{|c|}
\hline
\frac{4}{13}\\
\hline
\frac{2}{5}\\
\hline
\frac{72}{221}\\
\hline
\end{array} & 
\frac{1}{5} & 
\frac{24}{65} & 
\frac{48}{221} & 
\frac{24}{221} & 
\frac{5}{17} & 
\begin{array}{|c|}
\hline
\genfrac{}{}{0pt}{0}{t = 1}{s = 0}\\
\hline
t \geq 3\\
\hline
\genfrac{}{}{0pt}{0}{t = 1}{s \geq 2}\\
\hline
\end{array} \\ 
\hline
(c_{7}) &  
\frac{1}{2} & 
\frac{3}{16} & 
\frac{1}{8} & 
\frac{5}{16} & 
\frac{1}{8} & 
\frac{5}{17} & 
 \\ 
\hline
(c_{8}) &  
\frac{2}{5} & 
\frac{1}{5} & 
\frac{13}{24} & 
\begin{array}{|c|}
\hline
\frac{3}{17}\\
\hline
\frac{13}{60}\\
\hline
\end{array} & 
\frac{13}{120} & 
\frac{5}{17} & 
\begin{array}{|c|}
\hline
t = 0 \\
\hline
t \geq 2 \\
\hline
\end{array} \\ 
\hline
(c_{9}) &  
\frac{1}{2} & 
\frac{2}{13} & 
\frac{24}{221} & 
\frac{60}{221} & 
- & 
\frac{5}{17} & 
 \\ 
\hline
(c_{10}) &  
\frac{1}{2} & 
\frac{1}{5} & 
\frac{13}{120} & 
\frac{13}{48} & 
\begin{array}{|c|}
\hline
\frac{8}{85}\\
\hline
\frac{13}{120}\\
\hline
\frac{3}{34}\\
\hline
\frac{1}{12}\\
\hline
\end{array} & 
\frac{5}{17} & 
\begin{array}{|c|}
\hline
s, t \geq 3\\
\hline
\genfrac{}{}{0pt}{0}{s = 1}{t \geq 3}\\
\hline
\genfrac{}{}{0pt}{0}{s \geq 3}{t = 1}\\
\hline
s = t = 1\\
\hline
\end{array} \\ 
\hline
(c_{11}) &  
- & 
- & 
\frac{2}{3} & 
\frac{1}{3} & 
- & 
\frac{13}{36} & 
 \\ 
\hline
\hline
\text{Max} &  
\frac{1}{2} & 
\frac{1}{5} & 
\frac{2}{3} & 
\frac{1}{3} & 
\frac{1}{8} & 
\frac{13}{36} & 
 \\ 
\hline
\end{array}
$$
\caption{Types T1 - T6. \label{Table:1}}
\end{table}

\newpage

Table~\ref{Table:2} contains the case analysis for the following inequalities:
\begin{eqnarray}
& & \text{Type T7: } \frac{g(r, s-2, t +1)}{g(r, s, t)} \leq \frac{13}{72}, \label{Ineq:T7}\\
& & \text{Type T8: } \frac{g(r, s-2, t)}{g(r, s, t)} \leq \frac{1}{17}, \label{Ineq:T8}\\
& & \text{Type T9: } \frac{g(r-1, s+1, t)}{g(r, s, t)} \leq \frac{3}{4}, \label{Ineq:T9}\\
& & \text{Type T10: } \frac{g(r-1, s+1, t-1)}{g(r, s, t)}  \leq \frac{17}{48}, \label{Ineq:T10}\\
& & \text{Type T11: } \frac{g(r-1, s, t+1)}{g(r, s, t)} \leq \frac{5}{13}, \label{Ineq:T11}\\
& & \text{Type T12: } \frac{g(r-1, s, t)}{g(r, s, t)} \leq \frac{5}{26}. \label{Ineq:T12}
\end{eqnarray}

\begin{table}[H]
$$
\begin{array}{|c|c|c|c|c|c|c|l|}
\hline
\text{Case} & \text{\eqref{Ineq:T7}} & \text{\eqref{Ineq:T8}} & \text{\eqref{Ineq:T9}} & \text{\eqref{Ineq:T10}} & \text{\eqref{Ineq:T11}} & \text{\eqref{Ineq:T12}} & \\
\hline
(c_{1}) &  
\frac{2}{17} & 
\frac{1}{17} & 
\frac{24}{R} & 
\frac{48}{5 R} & 
\frac{13}{R} & 
\frac{26}{5 R} & 
 \\ 
\hline
(c_{2}) &  
\frac{5}{34} & 
\frac{1}{17} & 
\frac{24}{R} & 
\frac{48}{5 R} & 
\frac{13}{R} & 
\frac{26}{5 R} & 
 \\ 
\hline
(c_{3}) &  
\frac{8}{51} & 
\frac{1}{17} & 
\frac{24}{R} & 
- & 
- & 
- & 
 \\ 
\hline
(c_{4}) &  
- & 
- & 
\frac{3}{4} & 
- & 
- & 
- & 
 \\ 
\hline
(c_{5}) &  
\frac{2}{17} & 
\frac{1}{17} & 
\frac{221}{8R} & 
\begin{array}{|c|}
\hline
\frac{9}{R} \\
\hline
\frac{221}{20\cdot R} \\
\hline
\end{array} & 
\frac{15}{R} & 
\frac{6}{R} & 
\begin{array}{|c|}
\hline
t = 1 \\
\hline
t \geq 3 \\
\hline
\end{array} \\ 
\hline
(c_{6}) &  
\frac{2}{17} & 
\frac{1}{17} & 
\frac{8}{13} & 
\begin{array}{|c|}
\hline
\frac{3}{13} \\
\hline
\frac{16}{65} \\
\hline
\end{array} & 
\frac{5}{13} & 
\frac{2}{13} & 
\begin{array}{|c|}
\hline
t = 1 \\
\hline
t \geq 3 \\
\hline
\end{array} \\ 
\hline
(c_{7}) &  
\frac{5}{34} & 
\frac{1}{17} & 
\frac{221}{8R} & 
\frac{221}{16 R} & 
\frac{15}{R} & 
\frac{15}{2 R} & 
 \\ 
\hline
(c_{8}) &  
\frac{2}{17} & 
\frac{1}{17} & 
\frac{17}{24} & 
\frac{17}{60} & 
\frac{1}{3} & 
\frac{2}{15} & 
 \\ 
\hline
(c_{9}) &  
\frac{5}{34} & 
\frac{1}{17} & 
\frac{8}{13} & 
\frac{4}{13} & 
\frac{5}{13} & 
\frac{5}{26} & 
 \\ 
\hline
(c_{10}) &  
\frac{5}{34} & 
\frac{1}{17} & 
\frac{17}{24} & 
\frac{17}{48} & 
\frac{1}{15} & 
\frac{1}{6} & 
 \\ 
\hline
(c_{11}) &  
\frac{13}{72} & 
\begin{array}{|c|}
\hline
\frac{1}{18} \\
\hline
\end{array} & 
\frac{17}{24} & 
- & 
- & 
- & 
\begin{array}{|c|}
\hline
s = 2 \\
\hline
\end{array}  \\ 
\hline
\hline
\text{Max} &  
\frac{13}{72} & 
\frac{1}{17} & 
\frac{3}{4} & 
\frac{17}{48} & 
\frac{5}{13} & 
\frac{5}{26} & 
 \\ 
\hline
\end{array}
$$
\caption{Types T7 - T12. \label{Table:2}}
\end{table}


\newpage

Table~\ref{Table:3} contains the case analysis for the following inequalities:
\begin{eqnarray}
& & \text{Type T13: } \frac{g(r - 1, s, t - 1)}{g(r, s, t)} \leq \frac{1}{13}, \label{Ineq:T13}\\
& & \text{Type T14: } \frac{g(r - 1, s - 1, t + 2)}{g(r, s, t)} \leq \frac{5}{24}, \label{Ineq:T14}\\
& & \text{Type T15: } \frac{g(r - 1, s - 1, t + 1)}{g(r, s, t)} \leq \frac{1}{9}, \label{Ineq:T15}\\
& & \text{Type T16: } \frac{g(r - 1, s - 1, t)}{g(r, s, t)}  \leq \frac{1}{24}, \label{Ineq:T16}\\
& & \text{Type T17: } \frac{g(r - 2, s + 2, t)}{g(r, s, t)} \leq \frac{18}{R}, \label{Ineq:T17}\\
& & \text{Type T18: } \frac{g(r - 2, s + 1, t + 1)}{g(r, s, t)} \leq \frac{12}{R}. \label{Ineq:T18}
\end{eqnarray}

\begin{table}[H]
$$
\begin{array}{|c|c|c|c|c|c|c|l|}
\hline
\text{Case} & \text{\eqref{Ineq:T13}} & \text{\eqref{Ineq:T14}} & \text{\eqref{Ineq:T15}} & \text{\eqref{Ineq:T16}} & \text{\eqref{Ineq:T17}} & \text{\eqref{Ineq:T18}} & \\
\hline
(c_{1}) &  
\frac{13}{5R} & 
\frac{120}{17R} & 
\frac{48}{17R} & 
\frac{24}{17R} & 
\frac{17}{R} & 
\frac{8}{R} & 
 \\ 
\hline
(c_{2}) &  
\begin{array}{|c|}
\hline
\frac{13}{5R} \\
\hline
\frac{36}{17R} \\
\hline
\frac{2}{R}\\
\hline
\end{array} & 
\frac{120}{17R} & 
\frac{60}{17R} & 
\frac{24}{17R} & 
\frac{17}{R} & 
\frac{8}{R} & 
\begin{array}{|c|}
\hline
t \geq 1 \\
\hline
t = 1, s \geq 2 \\
\hline
t - 1 = s = 0 \\
\hline
\end{array} \\ 
\hline
(c_{3}) &  
- & 
\begin{array}{|c|}
\hline
\frac{120}{17R} \\
\hline
\frac{20}{3R} \\
\hline
\end{array} & 
\frac{13}{3R} & 
\begin{array}{|c|}
\hline
\frac{24}{17R} \\
\hline
\frac{4}{3R} \\
\hline
\end{array} & 
\frac{17}{R} & 
\frac{34}{3R} & 
\begin{array}{|c|}
\hline
s \geq 3 \\
\hline
s = 1 \\
\hline
\end{array} \\ 
\hline
(c_{4}) &  
- & 
- & 
- & 
- & 
\frac{18}{R} & 
\frac{12}{R} & 
 \\ 
\hline
(c_{5}) &  
\frac{3}{R} & 
\frac{65}{8R} & 
\frac{13}{4R} & 
\frac{13}{8R} & 
\frac{17}{R} & 
\frac{85}{8R} & 
 \\ 
\hline
(c_{6}) &  
\frac{1}{13} & 
\frac{40}{221} & 
\frac{16}{221} & 
\frac{8}{221} & 
\frac{17}{R} & 
\frac{120}{13R} & 
 \\ 
\hline
(c_{7}) &  
\frac{3}{R} & 
\frac{65}{8R} & 
\frac{65}{16R} & 
\frac{13}{8R} & 
\frac{17}{R} & 
\frac{85}{8R} & 
 \\ 
\hline
(c_{8}) &  
\frac{1}{15} & 
\frac{5}{24} & 
\frac{1}{12} & 
\frac{1}{24} & 
\frac{17}{R} & 
\frac{221}{24R} & 
 \\ 
\hline
(c_{9}) &  
\frac{1}{13} & 
\frac{40}{221} & 
\frac{20}{221} & 
\frac{8}{221} & 
\frac{17}{R} & 
\frac{120}{13R} & 
 \\ 
\hline
(c_{10}) &  
\begin{array}{|c|}
\hline
\frac{1}{15} \\
\hline
\frac{1}{16}\\
\hline
\end{array} & 
\frac{5}{24} & 
\frac{5}{48} & 
\frac{1}{24} & 
\frac{17}{R} & 
\frac{221}{24R} & 
\begin{array}{|c|}
\hline
t \geq 1 \\
\hline
t = 1\\
\hline
\end{array} \\ 
\hline
(c_{11}) &  
- & 
\frac{5}{24} & 
\frac{1}{9} & 
\frac{1}{24} & 
\frac{17}{R} & 
\frac{34}{3R} & 
 \\ 
\hline
\hline
\text{Max} &  
\frac{1}{13} & 
\frac{5}{24} & 
\frac{1}{9} & 
\frac{1}{24} & 
\frac{18}{R} & 
\frac{12}{R} & 
 \\ 
\hline
\end{array}
$$
\caption{Types T13 - T18. \label{Table:3}}
\end{table}

\newpage

Table~\ref{Table:4} contains the case analysis for the following inequalities:
\begin{eqnarray}
& & \text{Type T19: } \frac{g(r - 2, s + 1, t)}{g(r, s, t)} \leq \frac{6}{R}, \label{Ineq:T19}\\
& & \text{Type T20: } \frac{g(r - 2, s, t + 2)}{g(r, s, t)} \leq \frac{13}{2R}, \label{Ineq:T20}\\
& & \text{Type T21: } \frac{g(r - 2, s, t + 1)}{g(r, s, t)} \leq \frac{13}{4R}, \label{Ineq:T21}\\
& & \text{Type T22: } \frac{g(r - 2, s, t)}{g(r, s, t)}  \leq \frac{1}{R}. \label{Ineq:T22}
\end{eqnarray}

\begin{table}[H]
$$
\begin{array}{|c|c|c|c|c|l|}
\hline
\text{Case} & \text{\eqref{Ineq:T19}} & \text{\eqref{Ineq:T20}} & \text{\eqref{Ineq:T21}} & \text{\eqref{Ineq:T22}} & \\
\hline
(c_{1}) &  
\begin{array}{|c|}
\hline
\frac{16}{5R} \\
\hline
\frac{3}{R}\\
\hline
\end{array} & 
\frac{5}{R} & 
\frac{2}{R} & 
\frac{1}{R} & 
\begin{array}{|c|}
\hline
t \geq 2 \\
\hline
t = 0 \\
\hline
\end{array} \\ 
\hline
(c_{2}) &  
\frac{4}{R} & 
\frac{5}{R} & 
\frac{5}{2R} & 
\frac{1}{R} & 
 \\ 
\hline
(c_{3}) &  
\frac{17}{3R} & 
\frac{16}{3R} & 
\frac{8}{3R} & 
\frac{1}{R} & 
 \\ 
\hline
(c_{4}) &  
\frac{6}{R} & 
\frac{13}{2R} & 
\frac{13}{4R} & 
\frac{1}{R} & 
 \\ 
\hline
(c_{5}) &  
\frac{17}{4R} & 
\frac{5}{R} & 
\frac{2}{R} & 
\frac{1}{R} & 
 \\ 
\hline
(c_{6}) &  
\frac{48}{13R} & 
\frac{5}{R} & 
\frac{2}{R} & 
\frac{1}{R} & 
 \\ 
\hline
(c_{7}) &  
\frac{85}{16R} & 
\frac{5}{R} & 
\frac{1}{2R} & 
\frac{1}{R} & 
 \\ 
\hline
(c_{8}) &  
\begin{array}{|c|}
\hline
\frac{221}{60R} \\
\hline
\frac{3}{R} \\
\hline
\end{array} & 
\frac{5}{R} & 
\frac{2}{R} & 
\frac{1}{R} & 
\begin{array}{|c|}
\hline
t \geq 2 \\
\hline
t = 0 \\
\hline
\end{array}  \\ 
\hline
(c_{9}) &  
\frac{60}{13R} & 
\frac{5}{R} & 
\frac{1}{2R} & 
\frac{1}{R} & 
 \\ 
\hline
(c_{10}) &  
\frac{221}{48R} & 
\frac{5}{R} & 
\frac{1}{2R} & 
\frac{1}{R} & 
 \\ 
\hline
(c_{11}) &  
\frac{17}{3R} & 
\frac{221}{36R} & 
\frac{221}{72R} & 
\frac{1}{R} & 
 \\ 
\hline
\hline
\text{Max} &  
\frac{6}{R} & 
\frac{13}{2R} & 
\frac{13}{4R} & 
\frac{1}{R} & 
 \\ 
\hline
\end{array}
$$
\caption{Types T19 - T22. \label{Table:4}}
\end{table}

%% file: GR-K5-complete.bbl
\begin{thebibliography}{10}

\bibitem{MR729784}
F.~R.~K. Chung and R.~L. Graham.
\newblock Edge-colored complete graphs with precisely colored subgraphs.
\newblock {\em Combinatorica}, 3(3-4):315--324, 1983.

\bibitem{FGP15}
J.~Fox, A.~Grinshpun, and J.~Pach.
\newblock The {E}rd{\H o}s-{H}ajnal conjecture for rainbow triangles.
\newblock {\em J. Combin. Theory Ser. B}, 111:75--125, 2015.

\bibitem{FMO14}
S.~Fujita, C.~Magnant, and K.~Ozeki.
\newblock Rainbow generalizations of {R}amsey theory - a dynamic survey.
\newblock {\em Theo. Appl. Graphs}, 0(1), 2014.

\bibitem{MR0221974}
T.~Gallai.
\newblock Transitiv orientierbare {G}raphen.
\newblock {\em Acta Math. Acad. Sci. Hungar}, 18:25--66, 1967.

\bibitem{GG}
R.~E. Greenwood and A.~M. Gleason.
\newblock Combinatorial relations and chromatic graphs.
\newblock {\em Canad. J. Math.}, 7:1--7, 1955.

\bibitem{GSSS10}
A.~Gy{\'a}rf{\'a}s, G.~S{\'a}rk{\"o}zy, A.~Seb{\H o}, and S.~Selkow.
\newblock Ramsey-type results for gallai colorings.
\newblock {\em J. Graph Theory}, 64(3):233--243, 2010.

\bibitem{LMSSS17}
H.~Liu, C.~Magnant, A.~Saito, I.~Schiermeyer, and Y.~Shi.
\newblock Gallai-{R}amsey number for ${K}_{4}$.
\newblock {\em Submitted}.

\bibitem{M18}
C.~Magnant.
\newblock A general lower bound on {G}allai-{R}amsey numbers for non-bipartite
  graphs.
\newblock {\em Theo. and Appl. Graphs}, 5(1):Article 4, 2018.

\bibitem{MR1324481}
B.~D. McKay and S.~P. Radziszowski.
\newblock {$R(4,5)=25$}.
\newblock {\em J. Graph Theory}, 19(3):309--322, 1995.

\bibitem{MR1438619}
B.~D. McKay and S.~P. Radziszowski.
\newblock Subgraph counting identities and {R}amsey numbers.
\newblock {\em J. Combin. Theory Ser. B}, 69(2):193--209, 1997.

\bibitem{MR1670625}
S.~P. Radziszowski.
\newblock Small {R}amsey numbers.
\newblock {\em Electron. J. Combin.}, 1:Dynamic Survey 1, 30 pp. (electronic),
  1994.

\end{thebibliography}
